\documentclass[twoside,a4paper,reqno,11pt]{amsart}
\usepackage{amsfonts, amsbsy, amsmath, amsthm, amssymb, latexsym, verbatim, enumerate}
\usepackage{mathrsfs}
\usepackage[top=30mm,right=30mm,bottom=30mm,left=30mm]{geometry}

\usepackage{bm}
\usepackage[pdftex]{color,graphicx}

\makeatletter
\newcommand{\imod}[1]{\allowbreak\mkern4mu({\operator@font mod}\,\,#1)}
\makeatother

\renewcommand{\a}{\alpha}
\renewcommand{\b}{\beta}
\newcommand{\F}{\mathbb{F}}
 \newcommand{\e}{\epsilon}
\renewcommand{\bf}{\textbf} 
\renewcommand{\l}{\lambda} \renewcommand{\O}{\Omega}

 \newcommand{\C}{\mathcal{C}}

\newcommand{\leqs}{\leqslant}
\newcommand{\geqs}{\geqslant}

 \newcommand{\vs}{\vspace{3mm}}
\newcommand{\la}{\langle}
\newcommand{\ra}{\rangle}

\newcommand{\Hom}{\mathrm{Hom}}
\newcommand{\End}{\mathrm{End}}

\newtheorem{theorem}{Theorem}

\newtheorem{propn}[theorem]{Proposition}

\newtheorem{thm}{Theorem}[section]
\newtheorem{prop}[thm]{Proposition}
\newtheorem{lem}[thm]{Lemma}
\newtheorem{cor}[thm]{Corollary}

\newtheorem{rmk}[thm]{Remark}

\theoremstyle{definition}

\newtheorem{remk}[thm]{Remark}

\begin{document}

\author{Timothy C. Burness}
 \address{T.C. Burness, School of Mathematics, University of Bristol, Bristol BS8 1TW, UK}
 \email{t.burness@bristol.ac.uk}

\author{Martin W. Liebeck}
\address{M.W. Liebeck, Department of Mathematics,
    Imperial College, London SW7 2BZ, UK}
\email{m.liebeck@imperial.ac.uk}

\author{Aner Shalev}
\address{A. Shalev, Institute of Mathematics, Hebrew University, Jerusalem 91904, Israel}
\email{shalev@math.huji.ac.il}

\title[Generation of second maximal subgroups]{Generation of second maximal subgroups and the existence of special primes}

\thanks{The authors are grateful for the support of an EPSRC grant and for the hospitality of the Centre Interfacultaire Bernoulli at EPFL, where this work was completed.
  The third author acknowledges the support of Advanced ERC Grant 247034, an ISF grant 1117/13, and the Miriam and Julius Vinik Chair in Mathematics which he holds.}

\begin{abstract}
Let $G$ be a finite almost simple group. It is well known that $G$ can be generated by 3 elements, and in previous work we showed that 6 generators suffice for all maximal subgroups of $G$. In this paper we consider subgroups at the next level of the subgroup lattice -- the so-called second maximal subgroups. We prove that with the possible exception of some families of rank 1 groups of Lie type, the number of generators of every second maximal subgroup of $G$ is bounded by an absolute constant. We also show that such a bound holds without any exceptions if and only if there are only finitely many primes $r$ for which there is a prime power $q$ such that $(q^r-1)/(q-1)$ is prime. The latter statement is a formidable open problem in Number Theory.  Applications to random generation and polynomial growth  are also given.
\end{abstract}

\date{\today}
\maketitle

\setcounter{tocdepth}{1}
\tableofcontents

\newpage

\section{Introduction}\label{s:intro}

In recent years it has been shown that  finite non-abelian simple groups share several fundamental generation properties with their maximal subgroups. For example, both classes can be generated by a small number of elements -- the simple groups by 2 elements \cite{AG, St}, and their maximal subgroups by 4 elements \cite{BLS}. Similarly, both simple groups and their maximal subgroups are randomly generated by boundedly many elements \cite{BLS, LSh95}. Analogous results also hold for almost simple groups -- that is, groups lying between a non-abelian finite simple group and its automorphism group.
These groups are generated by $3$ elements \cite{DL} and their maximal subgroups by $6$ elements \cite{BLS}. 

In this paper we investigate analogous questions for subgroups lying deeper in the subgroup lattice of an almost simple group -- namely, for second maximal subgroups. We show, somewhat surprisingly, that the question of whether these subgroups are generated by a bounded number of elements is equivalent to a formidable open problem in Number Theory -- namely, the existence of primes of the form $\frac{q^r-1}{q-1}$ where $r$ is arbitrarily large and $q$ is a prime power (which may depend on $r$).

For a finite group $G$, let $d(G)$ be the minimal number of generators of $G$. Define the \emph{depth} of a subgroup $M$ of $G$ to be the maximal length of a chain of subgroups from $M$ to $G$. A subgroup is {\it second maximal} if it has depth 2.
There has been interest in the study of these subgroups and their overgroups in the context of lattice theory; this includes work of Feit \cite{feit83}, P\'{a}lfy \cite{Pal} and Aschbacher \cite{asch2}. In addition, the PhD thesis of Basile \cite{Basile} provides a detailed study of second maximal subgroups of symmetric and alternating groups.

Our first  result concerns the number of generators required for second maximal subgroups of almost simple groups.


\begin{theorem}\label{max}
Let $G$ be a finite almost simple group with socle $G_0$, and let $M$ be a second maximal subgroup of $G$. Then one of the following holds:
\begin{itemize}\addtolength{\itemsep}{0.2\baselineskip}
\item[{\rm (i)}] $d(M)\leqs 12$;
\item[{\rm (ii)}] $d(M) \leqs 70$, $G_0$ is exceptional of Lie type, and $M$ is maximal in a parabolic subgroup of $G$;
\item[{\rm (iii)}] $G_0 = {\rm L}_2(q)$, $^2{}B_2(q)$ or $^2{}G_2(q)$, and $M$ is maximal in a Borel subgroup of $G$.
\end{itemize}
\end{theorem}

The bounds 12 and 70 in parts (i) and (ii) are probably not best possible (see Remark \ref{tight}). In part (iii), $d(M)$ can be enormously large. For example, if $G = {\rm L}_{2}(2^k)$ and $2^k-1$ is a prime, then the elementary abelian $2$-group $M = (Z_2)^k$ is a second maximal subgroup of $G$ requiring $k$ generators. Since the largest currently known prime is a Mersenne prime with $k = 74207281$, we obtain the following.

\begin{propn}\label{mers}
There exists a second maximal subgroup $M$ of a finite simple group such that $d(M) = 74207281$.
\end{propn}

The question of whether $d(M)$ can be arbitrarily large for the groups in part (iii) of Theorem \ref{max} turns out to depend on the  open problem in Number Theory mentioned above:
\[\begin{array}{c}
\mbox{\emph{Are there infinitely many primes $r$ for which there}} \\
\mbox{\emph{exists a prime power $q$ such that $\frac{q^r-1}{q-1}$ is prime?}}
\end{array}\label{e:star} \tag{$\star$}\]

\noindent This would follow, for example, if there exist infinitely many Mersenne primes -- but note that in \eqref{e:star}, $q$ may be arbitrarily large and may depend on $r$. It is believed that question \eqref{e:star} has a positive answer. However, existing methods of Number Theory are far from proving this. 

We establish the following.

\vspace{1mm}

\begin{theorem}\label{open} The following are equivalent.
\begin{itemize}\addtolength{\itemsep}{0.2\baselineskip}
\item[{\rm (i)}] There exists a constant $c$ such that all second maximal subgroups of finite almost simple groups are generated by at most $c$ elements.
\item[{\rm (ii)}] There exists a constant $c$ such that all second maximal subgroups
of finite simple groups are generated by at most $c$ elements.
\item[{\rm (iii)}] There exists a constant $c$ such that all second maximal subgroups
of ${\rm L_2}(q)$ ($q$ a prime power) are generated by at most $c$ elements.
\item[{\rm (iv)}] The question \eqref{e:star} has a negative answer.
\end{itemize}
\end{theorem}

In view of the difficulty of question (\ref{e:star}), it seems likely that the validity of part (i) of Theorem \ref{open} will remain open for a long time. However, if we go further down the subgroup lattice and consider {\it third maximal} subgroups (i.e. subgroups of depth 3), we can show unconditionally that there is no bound on the number of generators:

\begin{propn}\label{unb}
For each real number $c$ there is a third maximal subgroup $M$ of  an almost simple group such that $d(M) > c$.
\end{propn}

Next we move on to random generation. For a finite group $G$ and a positive integer $k$ let
$P(G,k)$ denote the probability that $k$ randomly chosen elements
of $G$ generate $G$.
 Let $\nu(G)$ be the minimal number $k$ such that
$P(G,k) \geqs 1/e$. Up to a small multiplicative constant, it is known that
$\nu(G)$ is the expected number of random elements generating $G$
(see \cite{Pak} and \cite[Proposition 1.1]{lub}).
In \cite[Theorem 3]{BLS} it was shown that $\nu(M)$ is bounded by a constant for all maximal subgroups $M$ of almost simple groups.
Combining Theorem \ref{max} with results of Jaikin-Zapirain and Pyber \cite{JP}, we extend this to second maximal subgroups, as follows.

\begin{theorem}\label{nu}
There is a constant $c$ such that $\nu(M)\leqs c$ for all second maximal subgroups $M$ of almost simple groups, with the possible exception of those in part {\rm (iii)} of Theorem $\ref{max}$.
\end{theorem}

More precisely, we show that $\nu(M)\leqs c$ for every second maximal subgroup $M$ of an almost simple group if and only if the question \eqref{e:star} has a negative solution. Indeed, this follows by combining Theorem \ref{open} with Corollary \ref{all}.

Our final result concerns the growth of third maximal subgroups. Recall that for a group $G$ and a positive integer $n$, the number of maximal  subgroups of index $n$ in $G$ is denoted by $m_n(G)$. The maximal subgroup growth of finite and profinite groups has been widely studied in relation to the notion of positively finitely generated groups -- that  is, groups $G$ for which, for bounded $k$, $P(G,k)$ is bounded away from zero (see \cite{Ma, MS, lub}). For simple groups $G$, the theory was developed in \cite{KanL, LSh95}, culminating in \cite{LMSh}, where it was proved that $m_n(G) \leqs n^a$ for any fixed $a>1$ and sufficiently large $n$. A polynomial bound for second maximal subgroups was obtained in \cite[Corollary 6]{BLS}. This was based on the random generation of maximal subgroups by a bounded number of elements, together with Lubotzky's inequality $m_n(H) \leqs n^{\nu(H)+3.5}$ for all finite groups $H$ (\cite{lub}). Here we show that,  despite the fact that second maximal subgroups may not have such a random generation property, the growth of third maximal subgroups is still polynomial.

\begin{theorem}\label{poly}
There is a constant $c$ such that any almost simple group has at most $n^c$ third maximal subgroups of index $n$.
\end{theorem}

Our notation is fairly standard. We adopt the notation of \cite{KL} for classical groups, so ${\rm L}_{n}(q) = {\rm L}_{n}^{+}(q)$, ${\rm U}_{n}(q) = {\rm L}_{n}^{-}(q)$, ${\rm PSp}_{n}(q)$ and ${\rm P\O}_{n}^{\e}(q)$ denote the simple linear, unitary, symplectic and orthogonal groups of dimension $n$ over the finite field $\mathbb{F}_{q}$, respectively. In addition, we write $Z_n$ (or just $n$) and $D_n$ for the cyclic and dihedral groups of order $n$, respectively, and $[n]$ denotes an arbitrary solvable group of order $n$.

\vs

The paper is organised as follows. In Section \ref{s:prel} we start with some preliminary results that are needed in the proofs of our main theorems. Next, in Sections \ref{s:alt} and \ref{s:spor} we prove Theorem \ref{max} for groups with an alternating group and sporadic socle, respectively. This leaves us to deal with groups of Lie type. In Section \ref{s:class} we consider the non-parabolic subgroups of classical groups, and we do likewise for the exceptional groups in Section \ref{s:excep}. We complete the proof of Theorem \ref{max} in Section \ref{s:parab}, where we deal with the maximal subgroups of parabolic subgroups in groups of Lie type. Here we also present connections with Number Theory and the proof of Theorem \ref{open} is completed at the end of the section. Finally, in Section \ref{s:final} we discuss random generation and growth, and we prove Proposition \ref{unb} and  Theorems \ref{nu} and \ref{poly}.

\section{Preliminaries}\label{s:prel}

In this section we record several preliminary results that will be needed in the proofs of our main theorems. We start by recalling two of the main results from \cite{BLS}. The first is \cite[Theorem 2]{BLS}:

\begin{thm}\label{t:bls} {\rm (\cite{BLS}) }
Let $G$ be a finite almost simple group with socle $G_0$ and let $H$ be a maximal subgroup of $G$. Then $d(H \cap G_0) \leqs 4$ and $d(H) \leqs 6$.
\end{thm}

The next result is \cite[Theorem 7]{BLS}.

\begin{thm}\label{t:bls2} {\rm (\cite{BLS}) }
Let $G$ be a finite primitive permutation group with point stabilizer $H$. Then
$d(G)-1 \leqs d(H) \leqs d(G)+4$.
\end{thm}

Recall that if $M$ is a subgroup of a group $H$, then
$${\rm core}_H(M) = \bigcap_{h \in H}M^h$$
is the \emph{$H$-core} of $M$, which is the largest normal subgroup of $H$ contained in $M$.  The next result, which follows immediately from Theorems \ref{t:bls} and \ref{t:bls2}, will play a key role in our analysis of second maximal subgroups.

\begin{lem}\label{lem1}
Let $G$ be a finite almost simple group and let $M$ be a second maximal subgroup of $G$, so that $M<H<G$ with each subgroup maximal in the next. If ${\rm core}_H(M) =1$, then $d(M)\leqs 10$.
\end{lem}

\begin{proof}
Assume ${\rm core}_H(M) =1$, so $H$ acts faithfully and primitively on the cosets of $M$. Then $d(M) \leqs d(H)+4$ by Theorem \ref{t:bls2}, and $d(H) \leqs 6$ by Theorem \ref{t:bls}.
\end{proof}

\begin{remk}\label{clev}
In general, if $N = {\rm core}_H(M)$ then Lemma \ref{lem1} implies that $d(M/N)\leqs 10$, so
$$d(M)\leqs d_M(N)+d(M/N) \leqs d_M(N)+10$$
where $d_M(N)$ is the minimal number of generators of $N$ as a normal subgroup of $M$ (that is, $d_M(N)$ is the minimal $d$ such that $N = \la x_1^M, \ldots, x_d^M\ra$ for some  $x_i \in N$).
\end{remk}

\begin{lem}\label{aff}
Let $V$ be a finite dimensional vector space over $\F_q$, and let $G$ be a group such that ${\rm SL}(V) \leqs G \leqs {\rm \Gamma L}(V)$, ${\rm Sp}(V) \leqs G \leqs {\rm \Gamma Sp}(V)$ or $\Omega(V) \leqs G \leqs {\rm \Gamma O}(V)$. If $H$ is any maximal subgroup of $G$, then $V$ is a cyclic $\F_q(H\cap {\rm GL}(V))$-module.
\end{lem}

\begin{proof}
Set $\tilde{G}= G \cap {\rm GL}(V)$ and $\tilde{H} = H \cap {\rm GL}(V)$. The result is immediate if $\tilde{H}$ acts irreducibly on $V$, so let us assume $\tilde{H} = \tilde{G}_{U}$ is the stabilizer of a proper subspace $U$ of $V$. In the linear case, $\tilde{H}$ stabilizes no other proper non-zero subspace, so any vector $v \in V\setminus U$ generates $V$ as an $\mathbb{F}_q\tilde{H}$-module. Now assume $G$ is symplectic or orthogonal. If $U$ is totally singular (or a non-singular $1$-space when $G$ is orthogonal and $q$ is even) then any vector $v \in V\setminus U^{\perp}$ is a generator. Finally,  suppose $U$ is non-degenerate. Here $U$ and $U^{\perp}$ are the only proper non-zero $\tilde{H}$-invariant subspaces of $V$, so any vector $u_1+u_2 \in U \perp U^{\perp}$ with $u_1, u_2\ne 0$ is a generator.
\end{proof}

Suppose $G=S_n$ or $A_n$ and let $X=\mathbb{F}_{p}^n$ be the permutation module for $G$ over $\mathbb{F}_{p}$, where $n \geqs 3$ and $p$ is a prime. Set
$$U = \{(a_1, \ldots, a_n) \,:\, \sum a_i=0\},\;\; W = \{(a, \ldots, a) \,:\, a \in \mathbb{F}_{p}\}$$
and note that $W \subseteq U$ if $p$ divides $n$, otherwise $X = U \oplus W$.
It is easy to check that $U$ and $W$ are the only proper non-zero  submodules of $X$, so
the quotient $V = U/(U \cap W)$ is irreducible. We call $V$ the \emph{fully deleted permutation module} for $G$. Note that $\dim V = n-2$ if $p$ divides $n$, otherwise $\dim V = n-1$.

\begin{lem}\label{fdp}
Let $G=S_n$ or $A_n$, where $n \geqs 5$, let $p$ be a prime and let $V$ be the fully deleted permutation module for $G$ over $\mathbb{F}_{p}$. If $H$ is any maximal subgroup of $G$, then $V$ is a cyclic $\F_pH$-module.
\end{lem}

\begin{proof}
This is an easy exercise if $H$ is an intransitive or imprimitive maximal subgroup $(S_k\times S_{n-k})\cap G$ or
$(S_t \wr S_{n/t}) \cap G$. So assume now that $H$ is primitive on $I:= \{1,\ldots ,n\}$.
Let $\{e_1,\ldots,e_n\}$ be the standard basis of $\F_p^n$ and let $v = e_1-e_2 = (1,-1,0,\ldots,0) \in U$. We show that the orbit $v^H$ spans $U$. For a subset $J \subseteq I$, let $V(J) = \langle e_i-e_j \,: \, i,j \in J \rangle$. Note that $\langle v \rangle = V(\{1,2\})$.

Define $W$ to be the span of $v^H$. We claim that if $V(J) \subseteq W$ (where $1<|J|<n$) then there is a larger set $J'$ containing $J$ such that $V(J') \subseteq W$. To see this, note that as $H$ is primitive, $J$ is not a block for $H$, so there exists $h \in H$ such that $J \cap J^h$ is neither empty nor $J$. Say $h$ sends $i \mapsto x$, $j \mapsto y$, where $i,j \in J$, $x \in J$ and $y \not \in J$. Then $h$ sends $e_i-e_j \mapsto e_x-e_y$, and so $\langle V(J),
(e_i-e_j)^h \rangle$ contains $V(J')$, where $J' = J \cup \{ y\}$. Hence the claim, and the lemma follows.
\end{proof}

The next result concerns the minimal generation of maximal subgroups of certain wreath products. In the statement of the lemma, we use the notation $\frac{1}{e}H$ for a normal subgroup of index $e$ in $H$, and we write $V_4$ for the Klein four-group $Z_2 \times Z_2$.

\begin{lem}\label{l:wreath}
Let $G$ be one of the following groups, where $n \geqs 2$ and $A=S_n$ or $A_n$.
\begin{itemize}\addtolength{\itemsep}{0.2\baselineskip}
\item[{\rm (i)}] $G=\frac{1}{e}(Z_d \wr A)$, where $d \geqs 3$, $e$ divides $d$, and furthermore 
the natural projection map from $G$ to $A$ is surjective.
\item[{\rm (ii)}] $G=\frac{1}{e}(Z_2 \wr A)$ with $e=1$ or $2$.
\item[{\rm (iii)}] $G=\frac{1}{e}(V_4 \wr A)$ with $e=1, 2$ or $4$.
\item[{\rm (iv)}] $G=\frac{1}{e}(D_8 \wr A)$ with $e=1$ or $2$.
\item[{\rm (v)}] $G=\frac{1}{e}(Q_8 \wr A)$ with $e=1$ or $2$.
\end{itemize}
Then $d(H) \leqs 6$ for every maximal subgroup $H$ of $G$.
\end{lem}

\begin{proof}
The result is trivial for $n=2$ and for $n=3$, $A=A_3$, so assume that $n\geqs 3$ and $A \ne A_3$. First consider (i) and (ii). Without loss of generality, we may assume that $G = BA$, where the base group $B$ is the kernel of an $A$-invariant homomorphism from $(Z_d)^n$ to $Z_e$.
Then using the action of $A$ we see that $B = B(e)$, where
\[
B(e) = \{ (\l_1,\ldots ,\l_n) \in (Z_d)^n \, :\, \sum \l_i \equiv 0 \imod{e}\}
\]
(writing $Z_d$ as the additive group of integers modulo $d$). Let $H$ be a maximal subgroup of $G$.

Suppose first that $B \leqs H$. Then $H = BM$ where $M$ is a maximal subgroup of $A$.
As in the previous proof we see that there is a vector $v \in B$ such that $\langle v^M \rangle$ contains $B(0)$. Since $B(e)/B(0)$ is cyclic, it follows that $d_H(B) \leqs 2$ and thus $d(H) \leqs 2+d(M) \leqs 6$ since $d(M) \leqs 4$ by \cite[Proposition 4.2]{BLS}.

Now suppose that $B \not \leqs H$. Then $H/(H\cap B) \cong A$ and $H\cap B$ is a maximal $A$-invariant subgroup of $B$. Let $d = \prod p_i^{a_i}$ where the $p_i$ are distinct primes, and let $P_i$ be a Sylow $p_i$-subgroup of $B$. Order the $p_i$ so that $P_1\not \leqs H$. As each $P_i$ is $A$-invariant, we have
\[
H\cap B = (H\cap P_1)\,\prod_{i\geqs 2}P_i.
\]
Write $p=p_1$, $a=a_1$ and $p^b=e_p$ for the $p$-part of $e$, so that
\[
P_1 = \{(\l_1,\ldots ,\l_n) \in (Z_{p^a})^n \,:\, \sum \l_i \equiv 0 \imod{p^b}\}.
\]
Let $\phi: P_1 \rightarrow (Z_p)^n$ be the map sending $(\l_1,\ldots,\l_n)  \mapsto (p^{a-1}\l_1,\ldots,p^{a-1}\l_n)$.
Then $\phi(H\cap P_1)$ is a non-zero $A$-invariant subspace of $(Z_p)^n$, so is one of $U,W$ or $(Z_p)^n$ (where $U,W$ are as defined above). If $\phi(H\cap P_1)$ contains $U$, then  $H\cap P_1$ has an element $h$ of the form
\[
h = (1+p\l_1',-1+p\l_2',p\l_3',\ldots,p\l_n'),
\]
and $\langle h^A \rangle$ is a subgroup $(Z_{p^a})^{n-1}$ of $P_1$. Thus $d_H(H\cap P_1) \leqs 2$, and similarly $d_H(\prod_{i\geqs 2}P_i) \leqs 2$, so $d(H) \leqs d_H(H\cap B)+d(A) \leqs 6$. Finally, if $\phi(H\cap P_1) = W$ then $H\cap P_1 = \phi^{-1}(W)$ by maximality, and again we see that $d_H(H\cap P_1) \leqs 2$, giving the result as above.

The remaining cases are similar to but easier than (i) and (ii). Consider for example part (iv).  Let $B = G \cap (D_8)^n$ be the base group of $G$, and let $C = G \cap (Z_4)^n < B$. The result follows in the usual way if $B \leqs H$, so assume this is not the case. As in the proof of (i) we see that $d_H(H\cap C) \leqs 2$. Also $B/C = \frac{1}{e'}(Z_2)^n$, and we see in the usual way that $d_{H/H \cap C}(H\cap B/H\cap C) \leqs 2$. Hence $d(H) \leqs 4+d(A) \leqs 6$.
\end{proof}

We will also need some results on the generation of maximal subgroups of certain non-simple classical groups.

\begin{lem}\label{l:o4}
Let $G$ be a group such that $G_0 \leqs G \leqs {\rm Aut}(G_0)$, where $G_0 = {\rm P\O}_{4}^{+}(q)$, and let $H$ be a maximal subgroup of $G$. Then $d(G) \leqs 6$, $d(H) \leqs 8$ and $d(H \cap G_0) \leqs 4$.
\end{lem}

\begin{proof}
Here $G_0 = S \times S$ with $S={\rm L}_{2}(q)$ and it is easy to check that the result holds when $q \in \{2,3\}$. Now assume $q \geqs 4$, so $S$ is simple. Write $q=p^f$ with $p$ prime and set $H_0 = H \cap G_0$. Since $d(G_0)=2$ and every subgroup of
$${\rm Out}(G_0) = (Z_{(2,q-1)} \times Z_f) \wr S_2$$
is $4$-generator, it suffices to show that $d(H_0) \leqs 4$. Write $G = G_0.A$.

If $H$ contains $G_0$ then $H_0=G_0$ and thus $d(H_0)=2$. Otherwise $H = H_0.A$ and $H_0$ is a maximal $A$-invariant subgroup of $G_0$. It follows that $H_0$ is either a diagonal subgroup isomorphic to $S$, or it is of the form $S \times B$, $B \times S$, $B \times B$, where $B=C \cap S$ and $C$ is a maximal subgroup of an almost simple group with socle $S$. By inspecting \cite[Table 8.1]{BHR}, we observe that $d(B) \leqs 2$ and thus $d(H_0) \leqs 4$ as required.
\end{proof}

\begin{lem}\label{l:l2}
Let $G_0 \in \{{\rm L}_{2}(2), {\rm L}_{2}(3), {\rm U}_{3}(2)\}$ and let $H$ be a maximal subgroup of $G$, where $G_0 \leqs G \leqs {\rm Aut}(G_0)$. Then $d(H \cap G_0) \leqs 3$.
\end{lem}

\begin{proof}
This is a straightforward calculation.
\end{proof}

\section{Symmetric and alternating groups}\label{s:alt}

In this section we begin the proof of Theorem \ref{max} by handling the case where $G_0$ is an alternating group. Our main result is the following.

\begin{prop}\label{alt}
Let $G$ be an almost simple group with socle $A_n$. Then $d(M) \leqs 10$ for
every second maximal subgroup $M$ of $G$.
\end{prop}

\begin{proof}
If $n \leqs 8$ then it is easy to check that $d(M) \leqs 3$, so for the remainder we may assume that $G = A_n$ or $S_n$, with $n \geqs 9$. Write $M<H<G$, where $M$ is maximal in $H$, and $H$ is maximal in $G$. The possibilities for $H$ are given by the O'Nan-Scott theorem and we deduce that one of the following holds:
\begin{enumerate}\addtolength{\itemsep}{0.2\baselineskip}
\item[1.] $H$ is intransitive: $H=(S_{k} \times S_{n-k})\cap G$, $1 \leqs k<n/2$;
\item[2.] $H$ is affine: $H={\rm AGL}_{d}(p) \cap G$, $n=p^d$, $p$ prime, $d \geqs 1$;
\item[3.] $H$ is imprimitive or wreath-type: $H=(S_k \wr S_t)\cap G$, $n = kt$ or $k^t$;
\item[4.] $H$ is diagonal: $H=(T^k.({\rm Out}(T) \times S_k))\cap G$, $T$ non-abelian simple, $n = |T|^{k-1}$;
\item[5.] $H$ is almost simple.
\end{enumerate}

If $H$ is almost simple, then $d(M) \leqs 6$ by Theorem \ref{t:bls}, so we need to deal with the first four cases. Set $C={\rm core}_{H}(M)$. If $C=1$ then $d(M) \leqs 10$ by Lemma \ref{lem1}, so we may assume otherwise.

\vs

\noindent \emph{Case 1: $H$ is intransitive.}

\vs

First assume $k \geqs 5$. If $C$ contains $A_k \times A_{n-k}$ then \cite[Proposition 2.8]{BLS} implies that $d(M) \leqs 3$. Otherwise, $C$ and $H/C$ are $2$-generator almost simple groups and thus Theorem \ref{t:bls2} implies that
$$d(M) \leqs d(M/C)+d(C) \leqs d(H/C)+d(C)+4 \leqs 8.$$
Next suppose $k=4$. The result quickly follows if $C$ contains $V_4 \times A_{n-4}$, so assume otherwise. Then either $C$ is a subgroup of $S_4$ and $H/C$ has socle $A_{n-4}$, or vice versa, whence $d(M) \leqs 8$ as before. A very similar argument applies if $k \leqs 3$.

\vs

\noindent \emph{Case 2: $H$ is affine.}

\vs

Here $H = {\rm AGL}(V)\cap G = V.L$, where $V = \mathbb{F}_{p}^d$ is the unique minimal normal subgroup of $H$ and ${\rm SL}(V) \leqs L={\rm GL}(V) \cap G$. Note that $n=p^d$. Since we may assume $C \ne 1$ it follows that $M = V.J$ and $J<L$ is maximal. If $d=1$ or $(d,p)=(2,3)$ then it is easy to see that $d(M) \leqs 2$, so we may assume that ${\rm SL}(V)$ is quasisimple. Let $Z=Z(L)$ and note that $Z$ is cyclic. Then $L/Z$ is almost simple and thus $d(JZ/Z) \leqs 6$ by Theorem \ref{t:bls}. Therefore $d(J)\leqs 7$, and by applying Lemma \ref{aff} we deduce that $d(M) \leqs 8$.

\vs

\noindent \emph{Case 3: $H$ is imprimitive or wreath-type.}

\vs

First assume $G = S_n$. Write $H = S_k \wr S_t = N.S_t$, where $N=(S_k)^t$ and $k,t \geqs 2$. If $k=2$ then Lemma \ref{l:wreath} implies that $d(M) \leqs 6$, so we may assume that $k \geqs 3$.
Suppose $M$ contains $N$, so $M = N.J$ and $J<S_t$ is maximal. Now $J$ has $s \leqs 2$ orbits on $\{1, \ldots, t\}$, and $d(J) \leqs 4$ by \cite[Proposition 4.2]{BLS} (the cases with $t \leqs 4$ can be checked directly), so
$$d(M) \leqs d((S_k)^s)+d(J) \leqs 6$$
since $d(S_k \times S_k) = 2$ (see \cite[Proposition 2.8]{BLS}).
Now assume $M$ does not contain $N$, so $M=(M \cap N).S_t$ and $M \cap N$ is a maximal $S_t$-invariant subgroup of $N$. If $k \neq 4$ then $A=(A_k)^t$ is the unique minimal normal subgroup of $H$, so we may assume $H$ contains $A$ and we can consider $\bar{M}=M/A< \bar{H}=H/A = S_2\wr S_t$.
By Lemma \ref{l:wreath} we have $d(\bar{M}) \leqs 6$ and thus $d(M) \leqs d(A_k)+d(\bar{M}) \leqs 8$.
Similarly, if $k=4$ then $A=(V_4)^t$ is the unique minimal normal subgroup of $H$, so we may assume $M$ contains $A$. Note that $B=(Z_3)^t$ is the unique minimal normal subgroup of $H/A = S_3 \wr S_t$. If $M/A$ does not contain $B$, then Theorem \ref{t:bls2} implies that $d(M/A) \leqs d(H/A)+4=6$ and thus $d(M) \leqs 8$. Therefore, we may assume that $M/A$ contains $B$, so $M$ contains $(A_4)^t$ and the above argument goes through (via Lemma \ref{l:wreath}).

Now assume $G=A_n$ and $H = (S_k \wr S_t) \cap G$. If $H = S_k \wr S_t$ (which can happen if $n=k^t$) then the previous argument applies. Therefore, we may assume that $H$ is an index-two subgroup of $S_k \wr S_t$, so $H = ((A_k)^t.2^{t-1}).S_t$ or $S_k \wr A_t$. The latter case is handled as above, so let us assume $H = ((A_k)^t.2^{t-1}).S_t = N.S_t$. If $k=2$ then $H = \frac{1}{2}(S_2 \wr S_t)$ and thus $d(M) \leqs 6$ by Lemma \ref{l:wreath}. Now assume $k \geqs 3$.
If $M$ contains $N$ then $M=N.J$ with $J<S_t$ maximal and it is easy to see that $d(M) \leqs (2+1)s+d(J) \leqs 10$, where $s \leqs 2$ is the number of orbits of $J$ on $\{1, \ldots, t\}$. If $N \not\leqs M$ then we can reduce to the case where $M$ contains $(A_k)^t$ and by applying Lemma \ref{l:wreath} we deduce that $d(M) \leqs 8$.

\vs

\noindent \emph{Case 4: $H$ is diagonal.}

\vs

Write $H = (T^k.({\rm Out}(T) \times S_k)) \cap G$. First assume $G=S_n$. Here $H= T^k.({\rm Out}(T) \times S_k)$ and $T^k$ is the unique minimal normal subgroup of $H$, so $M = T^k.J$ for some maximal subgroup $J<{\rm Out}(T) \times S_k$. The projection of $J$ to the $S_k$ factor has $s \leqs 2$ orbits on $\{1, \ldots, k\}$ and thus $d(M) \leqs d(T^s)+d(J) =2+d(J)$. If $J$ is a standard maximal subgroup of ${\rm Out}(T) \times S_k$ (i.e. $J$ is of the form $A \times S_k$ or ${\rm Out}(T) \times B$, where $A,B$ are maximal in the respective factors), then $d(J) \leqs 7$ since every subgroup of ${\rm Out}(T)$ is $3$-generator, $d(S_k)\leqs 2$ and every maximal subgroup of $S_k$ is $4$-generator. The only other possibility is $J=(L \times A_k).2$, where $|{\rm Out}(T):L|=2$ (see \cite[Lemma 1.3]{Thev}, for example). Clearly, $d(J) \leqs 6$ in this case.

Now suppose $G=A_n$. We may as well assume that $H$ is an index-two subgroup of $T^k.({\rm Out}(T) \times S_k)$, otherwise the previous argument applies. If $k \geqs 3$, then $H = T^k.(L \times S_k)$, where $|{\rm Out}(T):L|=2$ (see the proof of \cite[Lemma 4.4]{BLS}), and $T^k$ is the unique minimal normal subgroup of $H$. In this situation, the above argument goes through unchanged. Finally, assume $k=2$. Set $\ell = \frac{1}{2}(|T|-i_2(T)-1)$, where $i_2(T)$ denotes the number of involutions in $T$. As explained in the proof of \cite[Lemma 4.4]{BLS}, if $\ell$ is even then $H = T^2.(L \times S_2)$ as above, and the usual argument applies. If $\ell$ is odd then $H = T^2.{\rm Out}(T)$, so $H$ has two minimal normal subgroups $N_1$ and $N_2$ (both isomorphic to $T$).
If $M$ contains $T^2$ then $M = T^2.J$ (with $J<{\rm Out}(T)$ maximal) and thus $d(M) \leqs d(T^2)+d(J) \leqs 2+3=5$. Otherwise we may assume that $M$ contains $N_1$, but not $N_2$, in which case $M/N_1$ is a maximal subgroup of $H/N_1 \cong {\rm Aut}(T)$. By Theorem \ref{t:bls} we have $d(M/N_1) \leqs 6$ and we conclude that $d(M) \leqs 8$.
\end{proof}

\section{Sporadic groups}\label{s:spor}

Our main result on second maximal subgroups of sporadic groups is the following.

\begin{prop}\label{spor}
Let $G$ be an almost simple group with sporadic socle $G_0$. Then $d(M) \leqs 10$ for
every second maximal subgroup $M$ of $G$.
\end{prop}

As before, write $M<H<G$ where $H$ is a maximal subgroup of $G$. Define sets $\mathcal{A}$ and $\mathcal{B}$ as follows:
\begin{align*}
\mathcal{A} & = \{{\rm M}_{11}, {\rm M}_{12},{\rm M}_{22},{\rm M}_{23},{\rm M}_{24},{\rm HS},
{\rm J}_{1}, {\rm J}_{2}, {\rm J}_{3}, {\rm Co}_{2},{\rm Co}_{3},{\rm McL},{\rm Suz},{\rm He},{\rm Fi}_{22},{\rm Ru} \} \\
\mathcal{B} & = \{{\rm O'N},{\rm J}_{4}, {\rm Th}, {\rm Ly}, {\rm HN}\}
\end{align*}

\begin{lem}\label{l:spor1}
If $G_0 \in \mathcal{A}\cup \mathcal{B}$, then $d(M) \leqs 5$.
\end{lem}

\begin{proof}
It is convenient to use {\sc Magma} \cite{Magma}, together with the detailed information on sporadic groups and their maximal subgroups provided in the Web-Atlas \cite{WebAt} . First assume $G_0 \in \mathcal{A}$. Here we use the Web-Atlas to construct $G$ as a permutation group of degree $n \leqs 6156$ (with equality if $G_0={\rm J}_{3}$), and we use the {\sc Magma} command \textsf{MaximalSubgroups} to construct $H$ and $M$ as permutation groups of degree $n$. In each case it is straightforward to find five generators for $M$ by random search.

A similar approach is effective if $G_0 \in \mathcal{B}$. For example, suppose $G={\rm O'N}.2$. First we use the Web-Atlas to construct $G$ as a permutation group on $245520$ points, and then we construct the maximal subgroups $H$ of $G$ using the generators given in the Web-Atlas. As before, we can use {\sc Magma} to find the maximal subgroups of $H$, and the desired result quickly follows. The remaining cases are similar, working with a suitable matrix representation when $G = {\rm J}_{4}$, ${\rm Th}$ or ${\rm Ly}$.
\end{proof}

\begin{remk}
The bound $d(M) \leqs 5$ in Lemma \ref{l:spor1} is sharp. For example, take
$$G={\rm Fi}_{22}.2, \;\; H={\rm U}_{4}(3).2^2 \times S_3,\;\; M = J.2^2 \times S_3,$$
where $J<{\rm U}_{4}(3)$ is a maximal subgroup of type ${\rm GU}_{2}(3) \wr S_2$. Then $M$ has a normal subgroup $N$ such that $M/N$ is elementary abelian of order $2^5$, so $d(M)=5$. (More precisely, $M = 2.{\rm L}_{2}(3)^2.2^4 \times S_3$ and $N = 2.{\rm L}_{2}(3)^2 \times 3$.)
\end{remk}

\begin{lem}\label{l:spor2}
If $G_0 \in \{{\rm Fi}_{23},  {\rm Fi}_{24}', {\rm Co}_{1}\}$, then $d(M) \leqs 8$.
\end{lem}

\begin{proof}
First observe that Theorem \ref{t:bls} implies that $d(M) \leqs 6$ if $H$ is almost simple, so we may assume otherwise.

Suppose $G={\rm Fi}_{23}$. Using the Web-Atlas, we construct $G$ as a permutation group of degree $31671$. In all but four cases, generators for $H$ are given in the Web-Atlas, and we can proceed as in the proof of the previous lemma. The exceptions are the following:
$$H \in \{3^{1+8}.2^{1+6}.3^{1+2}.2S_4, \; [3^{10}].({\rm L}_{3}(3) \times 2), \; 2^{6+8}{:}(A_7 \times S_3), \; {\rm Sp}_{6}(2) \times S_4\}.$$
It is easy to construct  $H={\rm Sp}_{6}(2) \times S_4$ as a permutation group of degree $67$, and the bound $d(M) \leqs 3$ quickly follows. We can obtain $H = 2^{6+8}{:}(A_7 \times S_3)$ as the normalizer in $G$ of a normal subgroup of order $2^{14}$ in a Sylow $2$-subgroup of $G$. The $3$-local subgroups $3^{1+8}.2^{1+6}.3^{1+2}.2S_4$ and $[3^{10}].({\rm L}_{3}(3) \times 2)$ can be constructed in a similar fashion, using a Sylow $3$-subgroup. In all three cases, it is easy to check that $d(M) \leqs 3$.

Next suppose $G = {\rm Fi}_{24}$. Here we start with a permutation representation of degree $306936$, and we construct the maximal subgroups $H$ of $G$ using the generators given in the Web-Atlas. In all but two cases, we can use {\sc Magma} to find the maximal subgroups $M$ of $H$, and verify the bound $d(M) \leqs 4$. The exceptions are the cases
$$H \in \{S_3 \times {\rm P\O}_{8}^{+}(3){:}S_3, \; {\rm Fi}_{23} \times 2\}.$$
If $H = {\rm Fi}_{23} \times 2$ then $M = {\rm Fi}_{23}$ or $J \times 2$, where $J<{\rm Fi}_{23}$ is maximal, so $d(M) \leqs 4$. Suppose $H = S_3 \times {\rm P\O}_{8}^{+}(3){:}S_3$. Then \cite[Lemma 1.3]{Thev} implies that
$$M \in \{J \times {\rm P\O}_{8}^{+}(3){:}S_3,\; S_3 \times L,\; (3 \times {\rm P\O}_{8}^{+}(3){:}3).2\},$$
where $J<S_3$ and $L<{\rm P\O}_{8}^{+}(3){:}S_3$ are maximal. Since $J$ is cyclic and $d(L) \leqs 6$, it follows that $d(M) \leqs 8$ in the first two cases. In the final case, it is clear that $d(M) \leqs 4$ (note that $d({\rm P\O}_{8}^{+}(3){:}3)=2$).

Now assume $G={\rm Fi}_{24}'$. Every maximal subgroup $H$ of $G$ is the intersection with $G$ of a maximal subgroup of $G.2= {\rm Fi}_{24}$ (with the exception of the almost simple maximal subgroups ${\rm He}{:}2$, ${\rm U}_{3}(3){:}2$ and ${\rm L}_{2}(13){:}2$), so we can construct $H$ as above, use {\sc Magma} to obtain the maximal subgroups $M$ of $H$, and then finally verify the desired bound on $d(M)$. This approach is effective unless $H = (3 \times {\rm P\O}_{8}^{+}(3){:}3).2$. Let $M$ be a maximal subgroup of $H$. If $M = 3 \times {\rm P\O}_{8}^{+}(3){:}3$ then clearly $d(M) \leqs 3$, so assume otherwise. Then $M = J.2$, and either $M \cong {\rm P\O}_{8}^{+}(3){:}S_3$ is almost simple, or $J = 3 \times L$ with $L = K \cap {\rm P\O}_{8}^{+}(3){:}3$ for some maximal subgroup $K<{\rm P\O}_{8}^{+}(3){:}S_3$. Since $d(L) \leqs 5$ by Theorem \ref{t:bls}, we conclude that $d(M) \leqs 7$.

Finally, suppose $G={\rm Co}_{1}$. Here we work with a permutation representation of degree $98280$. Explicit generators for the six largest maximal subgroups are given in the Web-Atlas, and in the usual way we deduce that $d(M) \leqs 3$. Representatives of the remaining sixteen conjugacy classes of maximal subgroups $H$ of $G$ can be constructed using the information provided in the Atlas \cite{Atlas} and Web-Atlas, and once again we find that $d(M) \leqs 3$. As before, the $p$-local maximal subgroups can be constructed by taking normalizers of appropriate normal subgroups of a Sylow $p$-subgroup of $G$. We leave the reader to check the details.
\end{proof}

\begin{lem}\label{l:sporb}
If $G = \mathbb{B}$ or $\mathbb{M}$, then $d(M) \leqs 10$.
\end{lem}

\begin{proof}
First assume $G = \mathbb{B}$. If $H$ is almost simple then $d(M) \leqs 6$, so we may assume otherwise. Also recall that $d(M) \leqs 10$ if ${\rm core}_{H}(M)=1$, so we may also assume that $M$ contains a nontrivial normal subgroup of $H$. The maximal subgroups of $G$ are listed in the Web-Atlas.

Suppose $H = 2.{}^2E_6(2){:}2$. Here $Z(H) \cong Z_2$ is the unique minimal normal subgroup of $H$, so we may assume that $M=2.J$, where $J<{}^2E_6(2){:}2$ is maximal. Since ${}^2E_6(2){:}2$ is almost simple, Theorem \ref{t:bls} implies that $d(M) \leqs 1+6=7$.

If $H = 2^{1+22}.{\rm Co}_{2}$ then $Z_2$ is the unique minimal normal subgroup of $H$, so we can assume that $M = Z_2 \times {\rm Co}_{2}$ or $2^{1+22}.J$, where $J<{\rm Co}_{2}$ is maximal. If $M = Z_2 \times {\rm Co}_{2}$ then $d(M) = 2$, so let us assume $M=2^{1+22}.J$. Here $Z_2$ is the unique minimal normal subgroup of $M$, so $d(M) = d(2^{22}.J)$ by \cite[Proposition 2.1(iii)]{BLS}. Using {\sc Magma}, we calculate that $J$ has at most $7$ composition factors on the irreducible $\mathbb{F}_{2}{\rm Co}_{2}$-module $2^{22}$, and by applying \cite[Proposition 3.1]{BLS} we conclude that $d(M) \leqs 7+d(J) \leqs 10$.

Next assume $H = (2^2 \times F_4(2)){:}2$. If $M =2^2 \times F_4(2)$ then $d(M) \leqs 4$, otherwise $M = (2 \times F_4(2)).2$ or $(2^2 \times J).2$, where $J = L \cap F_4(2)$ for some maximal subgroup $L<F_4(2).2$. In the first case it is clear that $d(M) \leqs 4$. In the latter,
Theorem \ref{t:bls} gives $d(J) \leqs 4$, so $d(M) \leqs 7$.

If $H=2^{2+10+20}.({\rm M}_{22}{:}2 \times S_3)$, $[2^{35}].(S_5 \times {\rm L}_{3}(2))$ or
$5^3.{\rm L}_{3}(5)$, then a permutation representation of $H$ of degree $6144$ is given in the Web-Atlas, and it is straightforward to show that $M$ is $3$-generator. Similarly, we can use a matrix representation of $H = 2^{9+16}.{\rm Sp}_{8}(2)$ of dimension $180$ over $\mathbb{F}_{2}$ to check that $d(M) \leqs 4$. The Web-Atlas also provides a matrix representation of $H=[2^{30}].{\rm L}_{5}(2)$ of dimension $144$ over $\mathbb{F}_{2}$ and one can check that $d(M)=2$ (we thank Eamonn O'Brien for his assistance with this computation).

In each of the remaining cases, we can take a suitable permutation representation of $H$ (see the proof of \cite[Proposition 3.3]{BOW}, for example), and it is straightforward to check that $d(M) \leqs 4$.

The case $G = \mathbb{M}$ is similar. Again we may assume that $H$ is not almost simple and ${\rm core}_{H}(M) \neq 1$, so $H$ belongs to one of the conjugacy classes of maximal subgroups of $G$ listed in the Web-Atlas. If $|H|<5\times 10^9$ then a permutation representation of $H$ is given in the Web-Atlas, and it is straightforward to check that $d(M) \leqs 4$. The remaining cases can be handled by arguing as above. For example, suppose $H = 2^{5+10+20}.(S_3 \times {\rm L}_{5}(2))$. Here $2^5$ is the unique minimal normal subgroup of $H$, so we may assume that $M = 2^{5+10+20}.J$ or $2^{5+10}.(S_3 \times {\rm L}_{5}(2))$, where $J<S_3 \times {\rm L}_{5}(2)$ is maximal. In the latter case, \cite[Proposition 2.1(iii)]{BLS} implies that $d(M) = d(S_3 \times {\rm L}_{5}(2)) = 2$. Now assume $M = 2^{5+10+20}.J$. Using {\sc Magma}, we calculate that $J$ has at most $8$ composition factors on $2^5$, $2^{10}$ and $2^{20}$, in total. Since every maximal subgroup of $S_3 \times {\rm L}_{5}(2)$ is $2$-generator, it follows that $d(M) \leqs 8+d(J) \leqs 10$.

Another possibility is $H = 3^8.{\rm P\O}_{8}^{-}(3).2_3$. In this case $3^8$ is the unique minimal normal subgroup of $H$, so we may assume that $M = 3^8.J$, where $J<{\rm P\O}_{8}^{-}(3).2_3$ is maximal. Here $3^8$ is the natural module for ${\rm P\O}_{8}^{-}(3)$ and Lemma \ref{aff} implies that $d(M) \leqs 1+d(J) \leqs 6$. The other cases are similar and we omit the details.
\end{proof}

\section{Classical groups}\label{s:class}

Let $G$ be an almost simple classical group over $\mathbb{F}_{q}$ with socle $G_0$, where $q=p^f$ for a prime $p$. Let $V$ be the natural $G_0$-module. Write $M<H<G$, where $M$ is maximal in $H$, and $H$ is maximal in $G$.

Let $n$ denote the dimension of $V$. Due to the existence of exceptional isomorphisms between certain low-dimensional classical groups (see \cite[Proposition 2.9.1]{KL}, for example), we may (and will) assume that $n \geqs 3$ if $G_0 = {\rm U}_{n}(q)$, $n \geqs 4$ if $G_0 = {\rm PSp}_{4}(q)'$, and $n \geqs 7$ if $G_0 = {\rm P\O}_{n}^{\e}(q)$. We also assume that $(n,q) \neq (4,2)$ if $G_0 = {\rm PSp}_{n}(q)'$, since ${\rm PSp}_{4}(2)' \cong A_6$.

The purpose of this section is to prove Theorem \ref{max} in the case where $G_0$ is classical and $M$ is contained in a maximal non-parabolic subgroup $H$ of $G$. It is convenient to postpone the analysis of maximal subgroups of parabolic subgroups to Section \ref{s:parab}, where we also deal with parabolic subgroups of exceptional groups.

By Aschbacher's subgroup structure theorem for finite classical groups (see \cite{asch}), with some exceptional cases for $G_0 = {\rm P\O}_8^+(q)$ or ${\rm PSp}_4(q)$ (with $q$ even), the maximal subgroup $H$ of $G$ is either almost simple, or it belongs to one of eight subgroup collections, denoted $\C_1, \ldots, \C_8$, which are roughly described in Table \ref{tab0}. In order to prove Theorem \ref{max} for classical groups, we will consider each of these subgroup collections in turn.

\begin{table}[h]
$$\begin{array}{ll} \hline \hline
\C_1 & \mbox{Stabilizers of subspaces of $V$} \\
\C_2 & \mbox{Stabilizers of decompositions $V=\bigoplus_{i}V_i$, where $\dim V_i  = a$} \\
\C_3 & \mbox{Stabilizers of prime index extension fields of $\mathbb{F}_{q}$} \\
\C_4 & \mbox{Stabilizers of decompositions $V=V_1 \otimes V_2$} \\
\C_5 & \mbox{Stabilizers of prime index subfields of $\mathbb{F}_{q}$} \\
\C_6 & \mbox{Normalizers of symplectic-type $r$-groups in absolutely irreducible representations} \\
\C_7 & \mbox{Stabilizers of decompositions $V=\bigotimes_{i}V_i$, where $\dim V_i  = a$} \\
\C_8 & \mbox{Stabilizers of non-degenerate forms on $V$} \\ \hline \hline
\end{array}$$
\caption{The $\C_i$ subgroup collections}
\label{tab0}
\end{table}

The main result of this section is the following.

\begin{prop}\label{cla}
Let $M$ be a second maximal subgroup of an almost simple classical group $G$ with socle $G_0$, where $M<H<G$ and $H$ is a maximal non-parabolic subgroup of $G$. Then $d(M) \leqs 12$.
\end{prop}

Set $M_0 = M\cap G_0$ and $H_0 = H\cap G_0$, and note that $G/G_0 \cong H/H_0$. If $M$ contains $H_0$ then $d(M) \leqs d(M_0)+d(M/H_0) \leqs d(M_0)+3$ since every subgroup of $G/G_0$ is $3$-generator. Otherwise $H/H_0 \cong M/M_0$ and again we deduce that $d(M) \leqs d(M_0)+3$. Therefore, it suffices to show that $d(M_0) \leqs 9$. This follows from Theorem \ref{t:bls} if $H$ is almost simple, so we may assume that $H$ belongs to one of the collections $\mathcal{C}_{i}$, $i = 1, \ldots, 8$ (or a small additional collection of maximal subgroups that arises when $G_0 = {\rm P\O}_{8}^+(q)$ or ${\rm PSp}_{4}(q)$ (with $q$ even)).

We begin with a useful preliminary result. Recall that the \emph{solvable residual} of a finite group is the smallest normal subgroup such that the respective quotient is solvable (equivalently, it is the last term in the derived series).

\begin{lem}\label{lem:sr}
Let $E=H^{\infty}$ be the solvable residual of $H$, and assume that $E$ is quasisimple and acts irreducibly on $V$. Then $d(M) \leqs 9$.
\end{lem}

\begin{proof}
Set $\tilde{G} = G \cap {\rm PGL}(V)$ and $\tilde{M} = M \cap {\rm PGL}(V)$. If $H$ contains $G_0$ then $H$ is almost simple and thus $d(M) \leqs 6$ by Theorem \ref{t:bls}, so assume otherwise. Set $C=C_G(E)$ and note that $C$ is a normal subgroup of $N_G(E)=H$. The irreducibility of $E$ on $V$ implies that $C_{\tilde{G}}(E)$ is cyclic, so $C_{\tilde{M}}(E)$ is also cyclic, and thus $d(M \cap C) \leqs 3$ since every subgroup of $G/\tilde{G}$ is $2$-generator. Therefore, in order to prove the lemma it suffices to show that $d(MC/C) \leqs 6$.

To see this, first note that $H/C$ is almost simple (with socle $EC/C \cong E/Z(E)$).
If $M$ contains $C$ then $MC/C=M/C$ is a maximal subgroup of $H/C$ and thus $d(MC/C) \leqs 6$ as required. On the other hand, if $M$ does not contain $C$ then $MC = H$, so $MC/C = H/C$ is almost simple and thus $d(MC/C) \leqs 3$.
\end{proof}

\begin{lem}\label{l:c1}
Proposition $\ref{cla}$ holds if $H \in \C_{1}$.
\end{lem}

\begin{proof}
The possibilities for $G$ and $H$ are listed in \cite[Table 4.1.A]{KL}. Recall that $H$ is non-parabolic.

First assume $G_0 = {\rm L}_{n}(q)$ and $H$ is of type ${\rm GL}_{m}(q) \oplus {\rm GL}_{n-m}(q)$, where $1 \leqs m < n/2$. It is convenient to work in the quasisimple group ${\rm SL}_{n}(q)$, so \cite[Proposition 4.1.4]{KL} implies that $H = N.A$ where $N = {\rm SL}_{m}(q) \times {\rm SL}_{n-m}(q)$ and $A \leqs (Z_{q-1} \times Z_{q-1}).(Z_f \times Z_2)$ with $q=p^f$. Note that $d(N)=2$ (see \cite[Proposition 2.5(ii)]{BLS}) and every subgroup of $(Z_{q-1} \times Z_{q-1}).(Z_f \times Z_2)$ is $4$-generator. In particular, if $M$ contains $N$ then $d(M) \leqs 6$, so assume otherwise. Then $M = (M \cap N).A$ and it suffices to show that $d(M \cap N) \leqs 8$. Since $M \cap N$ is a maximal $A$-invariant subgroup of $N$, it is of the form
$C \times {\rm SL}_{n-m}(q)$ or ${\rm SL}_{m}(q) \times D$, where $C = E \cap {\rm SL}_{m}(q)$ and $E$ is maximal in a group $F$ such that
${\rm SL}_{m}(q) \leqs F \leqs {\rm \Gamma L}_{m}(q).\la \gamma \ra$
(where $\gamma$ is a graph automorphism if $m \geqs 3$, otherwise $\gamma=1$), and similarly for $D$. Since $C$ and $D$ are $5$-generator by Theorem \ref{t:bls} (the cases $m=1$ and $(m,q)=(2,2)$ or $(2,3)$ can be checked directly), we conclude that $d(M \cap N) \leqs 7$ and the result follows.

A very similar argument applies if $G_0 = {\rm U}_{n}(q)$ and $H$ is of type ${\rm GU}_{m}(q) \perp {\rm GU}_{n-m}(q)$, and also if $G_0 = {\rm PSp}_{n}(q)$ and $H$ is of type ${\rm Sp}_{m}(q) \perp {\rm Sp}_{n-m}(q)$. We omit the details. To complete the proof, we may assume that $G_0 = {\rm P\O}_{n}^{\e}(q)$ and $n \geqs 7$. If $n,q$ are even and $H$ is of type ${\rm Sp}_{n-2}(q)$, then $H$ is almost simple and thus $d(M) \leqs 6$. Now assume that $H$ is of type $O_{m}^{\e_1}(q) \perp O_{n-m}^{\e_2}(q)$, where $(m,\e_1) \neq (n-m,\e_2)$. Note that $q$ is odd if $m$ or $n-m$ is odd. Again, it will be convenient to work in the quasisimple group $\O_{n}^{\e}(q)$.

If $m=1$ then $H_0 = \O_{n-1}(q).2$ and it is easy to see that $d(M_0) \leqs 5$. Now assume $m \geqs 2$, so \cite[Proposition 4.1.6]{KL} implies that $H_0=N.A$, where $N=
\O_{m}^{\e_1}(q) \times \O_{n-m}^{\e_2}(q)$ and $A=[2^i]$ with $i=1$ or $2$. Note that $N$ is $4$-generator. If $M$ contains $N$ then $d(M_0) \leqs 6$, so let us assume otherwise. Then $M_0 = (M \cap N).A$ and it suffices to show that $M \cap N$ is $7$-generator. If $\O_{m}^{\e_1}(q)$ and $\O_{n-m}^{\e_2}(q)$ are both quasisimple then we can repeat the argument in the first paragraph of the proof, using Theorem \ref{t:bls}, to deduce that $d(M \cap N) \leqs 7$. Therefore, we may assume that
$$\O_{m}^{\e_1}(q) \in \{\O_{2}^{\pm}(q), \O_3(3), \O_{4}^{+}(q)\}$$
and $\O_{n-m}^{\e_2}(q)$ is quasisimple (if $G_0 = \O_7(3)$ and $H$ is of type $O_{4}^{+}(3) \perp O_{3}(3)$, then it is easy to check that $d(M) \leqs 4$).
The first two cases are straightforward since $\O_{2}^{\pm}(q)$ is cyclic, and every subgroup of $\O_{3}(3)$ is $2$-generated. Finally, if $\O_{m}^{\e_1}(q) = \O_{4}^{+}(q)$
then the usual argument goes through, using Lemma \ref{l:o4} in place of Theorem \ref{t:bls}.
\end{proof}

\begin{lem}\label{l:c2}
Proposition $\ref{cla}$ holds if $H \in \C_{2}$.
\end{lem}

\begin{proof}
The various possibilities for $G$ and $H$ are recorded in \cite[Table 4.2.A]{KL}. First assume $G_0 = {\rm U}_{n}(q)$ and $H$ is of type ${\rm GL}_{n/2}(q^2)$. Here $n \geqs 4$ is even and it is convenient to work in the quasisimple group ${\rm SU}_{n}(q)$, so $H_0 = N.A$ where $N={\rm SL}_{n/2}(q^2)$ and $A=Z_{q-1}.Z_2$. If $M$ contains $N$ then $M_0 = N.B$ for some subgroup $B \leqs A$, whence $d(M_0) \leqs 4$ and the result follows. Otherwise, $M_0 = (M \cap N).A$ and $M \cap N =  C \cap {\rm SL}_{n/2}(q^2)$, where $C$ is a maximal subgroup of a group $D$ such that ${\rm SL}_{n/2}(q^2) \leqs D \leqs {\rm \Gamma L}_{n/2}(q^2)\langle \gamma \ra$
(here $\gamma$ is a graph automorphism if $n \geqs 6$, otherwise $\gamma=1$). Therefore Theorem \ref{t:bls} implies that $d(M \cap N) \leqs 4$ and thus $d(M_0) \leqs 6$.

Next suppose $G_0 = {\rm P\O}_{n}^{\e}(q)$ and $H$ is of type $O_{n/2}(q)^2$. Here $qn/2$ is odd, $n \geqs 10$ and $H=N.A$, where $N = \O_{n/2}(q) \times \O_{n/2}(q)$ and $A \leqs [2^4].Z_f$. Note that $d(N)=2$. If $M$ contains $N$ then $M = N.B$ with $B \leqs A$ and thus $d(M) \leqs 7$, so assume otherwise. Then $M = (M \cap N).A$ and $M \cap N$ is a maximal $A$-invariant subgroup of $N$. If $M$ does not contain an element that interchanges the two $\O_{n/2}(q)$ factors of $N$ then $M \cap N = C \times \O_{n/2}(q)$ or $\O_{n/2}(q) \times C$, where $C = D \cap \O_{n/2}(q)$ and $D$ is maximal in an almost simple group with socle $\O_{n/2}(q)$. By Theorem \ref{t:bls} we have $d(C) \leqs 4$, so $d(M \cap N) \leqs 6$ and thus $d(M) \leqs 11$. Now assume $M$ has an element that interchanges the two $\O_{n/2}(q)$ factors. Then either $M \cap N$ is a diagonal subgroup isomorphic to $\O_{n/2}(q)$, or $M \cap N = C \times C$ with $C$ as above. In the former case, $d(M \cap N)=2$ and therefore $d(M) \leqs 7$. Finally, suppose $M \cap N = C \times C$. Write $M_0 = (M \cap N).B$ with $B \leqs Z_2 \times Z_2$. To obtain a generating set for $M$, take $4$ generators for one of the factors $C$, take an element in $M$ that swaps the two $\O_{n/2}(q)$ factors, and take two generators for $B$. These $7$ elements generate a subgroup $M_0.2 \leqs M$ and we can obtain a set of generators for $M$ by choosing at most two further elements. We conclude that $d(M) \leqs 9$.

Similar arguments apply in each of the remaining cases. For brevity, we only provide details in the two most difficult cases:
\begin{itemize}\addtolength{\itemsep}{0.2\baselineskip}
\item[(a)] $G_0 = {\rm L}_{n}^{\e}(q)$ and $H$ is of type ${\rm GL}_{a}^{\e}(q) \wr S_t$;
\item[(b)] $G_0 = {\rm P\O}_{n}^{\e}(q)$ and $H$ is of type $O_{a}^{\e'}(q) \wr S_t$, where $a \geqs 2$ is even and $q$ is odd.
\end{itemize}

Consider case (a). To begin with, let us assume $\e=+$ and $a \geqs 2$ (the special case $a=1$ will be handled later). Note that $(a,q) \neq (2,2)$ (see \cite{BHR,KL}). Set $d=(a,q-1)$. By \cite[Proposition 4.2.9]{KL} we have $H_0=N_0.S_t$ and $H = N.S_t$, where $N_0 = A_0.B_0$, $N=A.B$ such that $A_0$ and $A$ are sections of $(Z_{q-1})^t$, $B_0 = {\rm L}_{a}(q)^t.\frac{1}{d}(Z_d)^{t}$ and $B = {\rm L}_{a}(q)^t.C.2^b.Z_k$ where $C = \frac{1}{e}(Z_d)^t$ for some divisor $e$ of $d$, $b \in \{0,1\}$ and $k$ is a divisor of $\log_pq$. Write ${\rm L}_{a}(q) = \la x,y\ra$, where $x$ and $y$ have coprime orders (see \cite[Proposition 2.11]{BLS}), and fix $\delta$ such that ${\rm PGL}_{a}(q) = {\rm L}_{a}(q).\la \delta \ra$. Also write $\mathbb{F}_{q}^{\times} = \la \l \ra$ and fix an element $\mu \in \mathbb{F}_{q}^{\times}$ of order $d$.

Suppose $M$ contains $N$. Then $M=N.J$ with $J<S_t$ maximal, hence $M_0 = N_0.J$. If $J$ is transitive on $\{1, \ldots, t\}$ then $M_0$ is a quotient of the subgroup of ${\rm GL}_{a}(q) \wr J$ generated by the elements $(\l,\l^{-1}, 1, \ldots , 1)$, $(\mu, 1, \ldots, 1)$, $(x,y,1,\ldots, 1)$ and $(\delta, \delta^{-1}, 1, \ldots, 1)$ in ${\rm GL}_{a}(q)^t$, plus at most four generators for $J$, whence $d(M_0) \leqs 8$ and the result follows. Similarly, if $J$ is intransitive then $d(J) \leqs 2$ and once again we deduce that $d(M_0) \leqs 8$.

Now assume $N \not\leqs M$, so $M = (M \cap N).S_t$ and $M \cap N$ is a maximal $S_t$-invariant subgroup of $N$. Suppose $A$ is not contained in $M$. Then $M = (M \cap A).B.S_t$ and $M \cap A$ is a maximal $S_t$-invariant subgroup of $A$. In other words, $(M \cap A).S_t$ is a maximal subgroup of $A.S_t$. Since $A.S_t$ is a quotient of a group of the form $\frac{1}{s}(Z_{q-1} \wr S_t)$ for some divisor $s$ of $q-1$, Lemma \ref{l:wreath} implies that $d((M \cap A).S_t) \leqs 6$ and we deduce that $d(M) \leqs 11$. Now assume $M$ contains $A$. Set $\bar{M}=M/A$, $\bar{H}=H/A = B.S_t$ and let us assume that $(a,q) \ne (2,3)$. Here $S = {\rm L}_{a}(q)^t$ is the unique minimal normal subgroup of $\bar{H}$, so we may assume that $\bar{M}$ contains $S$ (if not, then Theorem \ref{t:bls2} implies that $d(\bar{M}) \leqs 10$ and thus $d(M) \leqs 12$ since $A$ is $2$-generator as a normal subgroup of $M$). We now consider the quotient groups $\tilde{M} = \bar{M}/S$ and $\tilde{H}=\bar{H}/S=C.2^b.Z_k.S_t$. If $\tilde{M}$ does not contain $C = \frac{1}{e}(Z_d)^t$ then $\tilde{M} = (\tilde{M} \cap C).2^b.Z_k.S_t$ and $(\tilde{M} \cap C).S_t<C.S_t$ is maximal. Now Lemma \ref{l:wreath} implies that $(\tilde{M} \cap C).S_t$ is $6$-generator, hence $d(\tilde{M}) \leqs 8$ so $d(\bar{M}) \leqs 9$ and thus $d(M) \leqs 11$. We have now reduced to the case where $C \leqs \tilde{M}$, hence $M_0 = H_0$ and Theorem \ref{t:bls} implies that $d(M_0) \leqs 4$.

Now assume $\e=+$ and $(a,q) = (2,3)$. As above, we may assume that $M = (M \cap N).S_t$ contains $A$, but the rest of the argument needs to be slightly modified since ${\rm L}_{2}(3) = A_4 = V_4{:}3$ is not simple. Set $\bar{M}=M/A$ and $\bar{H}=H/A = B.S_t$, where $B = (A_4)^t.C.2^b$. Now $D=(V_4)^t$ is the unique minimal normal subgroup of $\bar{H}$, so we may as well assume it is contained in $\bar{M}$. Set  $\tilde{M}=\bar{M}/D$ and $\tilde{H}=\bar{H}/D=E.C.2^b.S_t$ with $E=3^t$. If $\tilde{M}$ does not contain $E$ then $\tilde{M} = (\tilde{M} \cap E).C.2^b.S_t$ and $(\tilde{M} \cap E).S_t<E.S_t$ is maximal, so $\tilde{M} \cap E = 3$ or $3^{t-1}$ and $(\tilde{M} \cap E).S_t$ is $3$-generator. It follows that $d(\tilde{M}) \leqs 6$, so $d(\bar{M}) \leqs 8$ and $d(M) \leqs 10$.
We have now reduced to the case where $E \leqs \tilde{M}$, so $S=(A_4)^t \leqs M$ and the remainder of the previous argument now goes through.

To complete the analysis of the case $\e=+$, we may assume that $a=1$. Here $q \geqs 5$ and $H=N.S_t$, where $N = A.2^b.Z_k.S_t$ with $A$, $b$ and $k$ as above. It is easy to reduce to the case where $M = (M \cap N).S_t$. If $M$ contains $A$ then $M_0=H_0$ is $4$-generator, so assume otherwise. Then $M = (M \cap A).2^b.Z_k.S_t$ and $(M \cap A).S_t < A.S_t$ is maximal. This is a situation we considered above, and by applying Lemma \ref{l:wreath} we deduce that $d(M) \leqs 8$.

A similar argument applies when $\e=-$, so we will only give details in the special case $(a,q)=(3,2)$. Here ${\rm U}_{3}(2) = 3^2{:}Q_8$, $H_0 = N_0.S_t$ and $H=N.S_t$, where $N_0=A_0.B_0$, $N=A.B$ such that $A_0$ and $A$ are sections of $(Z_3)^t$, $B_0 = {\rm U}_{3}(2)^t.\frac{1}{3}(Z_3)^t$ and $B = {\rm U}_{3}(2)^t.C.2^b$ where $C = \frac{1}{3}(Z_3)^t$ or $(Z_3)^t$ and $b \in \{0,1\}$.
It is straightforward to reduce to the case where $M = (M \cap N).S_t$, and by arguing as above we may assume that $M$ contains $A$. Set $\bar{M}=M/A$ and $\bar{H}=H/A = (3^2{:}Q_8)^t.C.2^b.S_t$. Now $D=(3^2)^t$ is the unique minimal normal subgroup of $\bar{H}$, so we may assume that $\bar{M}$ contains $D$. Now set $\tilde{M}=\bar{M}/D$ and $\tilde{H}=\bar{H}/D = (Q_8)^t.C.2^b.S_t$. Let $E=(Q_8)^t$. If $E$ is not contained in $\tilde{M}$ then $\tilde{M} = (\tilde{M} \cap E).C.2^b.S_t$ and $(\tilde{M} \cap E).S_t<E.S_t$ is maximal, so Lemma \ref{l:wreath} implies that $(\tilde{M} \cap E).S_t$ is $6$-generator and we deduce that $d(\bar{M}) \leqs 9$ and $d(M) \leqs 11$. On the other hand, if $E \leqs \tilde{M}$ then $M$ contains ${\rm U}_{3}(2)^t$ and we can complete the proof as above.

\vs

Finally, let us turn to case (b). Let $D$ and $D'$ denote the discriminants of the quadratic forms corresponding to
$O_n^{\e}(q)$ and $O_{a}^{\e'}(q)$ (see \cite[p.32]{KL}, for example). To begin with, we will assume that $a \geqs 4$ and $(a,\e') \ne (4,+)$.

First assume $D' = \square$. By \cite[Proposition 2.11]{BLS} we have ${\rm P\O}_{a}^{\e'}(q)=\la x, y\ra$, where $|x|$ and $|y|$ are coprime. Fix involutions $r$ and $s$ such that ${\rm PSO}_{a}^{\e'}(q) = {\rm P\O}_{a}^{\e'}(q).\la s \ra$, ${\rm PO}_{a}^{\e'}(q) = {\rm PSO}_{a}^{\e'}(q).\la r \ra$ and $[r,s]=1$, so $\la r,s \ra = V_4$. By \cite[Proposition 4.2.11]{KL} we have $H_0 = N_0.S_t$ and $H=N.S_t$ where
\begin{equation}\label{e:NN}
N_0 = 2^{t-1}.{\rm P\O}_{a}^{\e'}(q)^t.2^{2(t-1)},\;\; N = 2^{t-1}.{\rm P\O}_{a}^{\e'}(q)^t.2^{2(t-1)}.[2^b].Z_k
\end{equation}
with $0 \leqs b \leqs 3$ and $k$ a divisor of $\log_pq$. Note that $[2^b] \leqs D_8$ is $2$-generator.

Suppose $M$ contains $N$, so $M_0 = N_0.J$ for some maximal subgroup $J<S_t$. If $J$ is transitive then $M_0$ is generated by $(-1,1,\ldots, 1)$, $(x,y,1,\ldots, 1)$, $(r,r,1,\ldots, 1)$ and $(s,s,1,\ldots, 1)$, together with at most $4$ more for $J$. This gives $d(M_0) \leqs 8$. Similarly, if $J$ is intransitive then $d(J) \leqs 2$ and we need at most $8$ generators for $M_0$.

Now assume $N \not\leqs M$, so $M=(M \cap N).S_t$ and $M \cap N$ is a maximal $S_t$-invariant subgroup of $N$. Write $N=A.B$, where $A=2^{t-1}$ and $B={\rm P\O}_{a}^{\e'}(q)^t.2^{2(t-1)}.[2^b].Z_k$. If $A \not\leqs M$ then $M= (M \cap A).B.S_t$ and $M \cap A$ is a maximal $S_t$-invariant subgroup of $A$. Therefore $M \cap A = 2^v$ with $v \in \{0,t-2\}$, so  $d((M \cap A).S_t) \leqs 3$ and we deduce that $d(M) \leqs 9$. Now assume $M$ contains $A$ and set $\bar{M}=M/A$ and $\bar{H}=H/A$. Here $S={\rm P\O}_{a}^{\e'}(q)^t$ is the unique minimal normal subgroup of $\bar{H}$, so we may assume that $S \leqs \bar{M}$ (if not, Theorem \ref{t:bls2} implies that $d(\bar{M}) \leqs d(\bar{H})+4=10$ and thus $d(M) \leqs 11$). Set $\tilde{M} =\bar{M}/S$ and $\tilde{H}=\bar{H}/S=C.[2^b].Z_k.S_t$, where $C=2^{2(t-1)}$. If $\tilde{M}$ contains $C$ then $M_0=H_0$ and thus $d(M_0) \leqs 4$, so assume otherwise. Then $\tilde{M} = (\tilde{M} \cap C).[2^b].Z_k.S_t$ and $(\tilde{M} \cap C).S_t < \frac{1}{4}(V_4 \wr S_t)$ is maximal. Since $(\tilde{M} \cap C).S_t$ is $6$-generator by Lemma \ref{l:wreath}, we conclude that $d(\tilde{M}) \leqs 6+2+1=9$, so $d(\bar{M}) \leqs 10$ and thus $d(M) \leqs 11$ as required.

Next suppose that $D' = \boxtimes$, so ${\rm P\O}_{a}^{\e'}(q) = \O_{a}^{\e'}(q) = {\rm PSO}_{a}^{\e'}(q)$. We continue to assume that $a \geqs 4$ and $(a,\e') \ne (4,+)$. By \cite[Proposition 4.2.11]{KL} we have $H_0 = N_0.S_t$ and $H=N.S_t$ where
$$N_0 = 2^{d} \times \O_{a}^{\e'}(q)^t.2^{t-1},\;\; N = 2^{e} \times \O_{a}^{\e'}(q)^t.2^{b}.2^c.Z_k$$
with $b \in \{t-1,t\}$, $c \in \{0,1\}$ and $k$ a divisor of $\log_pq$. Also, $d=e=t-1$ if $t$ is odd, otherwise $d=t-2$ and $e \in \{t-2,t-1\}$. Define the elements $x,y$ and $r$ as above. It is straightforward to reduce to the case where $M = (M \cap N).S_t$.

Write $N = A \times B$, where $A=2^{e}$ and $B=\O_{a}^{\e'}(q)^t.2^{b}.2^c.Z_k$. If $M$ contains $A$ then we may assume that $\bar{M}=M/A$ contains $\O_{a}^{\e'}(q)^t$, which is the unique minimal normal subgroup of $\bar{H}=H/A$. Therefore, $M_0 = (2^{d} \times \O_{a}^{\e'}(q)^t.2^v).S_t$ and the $S_t$-invariance of $M_0 \cap N$ implies that $v \in \{0,1,t-1\}$, so $d(M_0) \leqs 5$. Now assume $A \not\leqs M$, so $M =  (M \cap A).B.S_t$ and $M \cap A$ is a maximal $S_t$-invariant subgroup of $A$. Therefore $M \cap A = 2^v$ with $v \in \{0,t-2\}$, so $d((M \cap A).S_t) \leqs 3$ and $d(M) \leqs 7$.

To complete the proof, we may assume that $(a,\e')=(4,+)$ or $a=2$. Suppose $(a,\e')=(4,+)$.
Define the involutions $r$ and $s$ as above and note that $D'=\square$, ${\rm P\O}_{4}^{+}(q) = {\rm L}_{2}(q) \times {\rm L}_{2}(q)$ and ${\rm P\O}_{4}^{+}(q). \la r \ra = {\rm L}_{2}(q) \wr S_2$.  If $q \geqs 5$ then we can still write ${\rm P\O}_{4}^{+}(q) = \la x,y \ra$, where $|x|$ and $|y|$ are coprime, but this is not possible when $q=3$ (note that ${\rm P\O}_{4}^{+}(3) = A_4 \times A_4$ can be generated by $x$ and $y$, where $|x|=6$ and $|y|=3$). As above, we have $H_0 = N_0.S_t$ and $H = N.S_t$, where $N_0$ and $N$ are given in \eqref{e:NN}. One now checks that the argument above goes through essentially unchanged. Indeed, the only difference is that if $q=3$ then we require two generators for ${\rm P\O}_{4}^{+}(q)^t$ as a normal subgroup of ${\rm P\O}_{4}^{+}(q)^t.S_t$, rather than one. However, it is clear that the desired bound $d(M) \leqs 12$ still holds in this case. For example, if $q=3$ and $M$ contains $N$ then we get $d(M_0) \leqs 9$ and the result follows.

Finally, suppose $a=2$. We will assume $q \equiv \e' \imod{4}$ (the other case is very similar), so $D' = \square$ and ${\rm P\O}_{2}^{\e'}(q) = Z_{m}$ is cyclic, where $m=(q-\e')/4$. Write $H=N.S_t$, where $N=A.B$, $A=2^{t-1}$ and $B = (Z_m)^t.2^{2(t-1)}.[2^b].Z_k$ with $0 \leqs b \leqs 3$ and $k$ a divisor of $\log_pq$. In the usual way, we reduce to the case where $M = (M \cap N).S_t$. If $M$ does not contain $A$ then $M = (M \cap A).B.S_t$, where $M \cap A$ is a maximal $S_t$-invariant subgroup of $A$, so $M \cap A = 2^v$ with $v \in \{0,t-2\}$. Therefore, $d((M \cap A).S_t) \leqs 3$ and we deduce that $d(M) \leqs 9$. Now assume $M$ contains $A$. Set $\bar{M}=M/A$ and $\bar{H}=H/A = C.2^{2(t-1)}.[2^b].Z_k.S_t$, where $C = (Z_m)^t$. If $C \not\leqs \bar{M}$ then $\bar{M} = (\bar{M} \cap C).2^{2(t-1)}.[2^b].Z_k.S_t$ and $(\bar{M} \cap C).S_t<Z_m \wr S_t$ is maximal. Therefore Lemma \ref{l:wreath} implies that $d((\bar{M} \cap C).S_t) \leqs 6$, so $d(\bar{M}) \leqs 11$ and thus $d(M) \leqs 12$. Now assume $C \leqs \bar{M}$ and set $\tilde{M}=\bar{M}/C$ and $\tilde{H}=\bar{H}/C=D.[2^b].Z_k.S_t$, where $D=2^{2(t-1)}$. If $D \not\leqs \tilde{M}$ then $\tilde{M} = (\tilde{M} \cap D).[2^b].Z_k.S_t$ and $(\tilde{M} \cap D).S_t$ is a maximal subgroup of $D.S_t = \frac{1}{4}(V_4 \wr S_t)$. By Lemma \ref{l:wreath} we have $d((\tilde{M} \cap D).S_t) \leqs 6$, so $d(\tilde{M}) \leqs 9$, $d(\bar{M}) \leqs 10$ and thus $d(M) \leqs 11$. Finally, if $\tilde{M}$ contains $D$ then $M_0=H_0$ and $d(M_0) \leqs 4$.
\end{proof}

\begin{lem}\label{l:c3}
Proposition $\ref{cla}$ holds if $H \in \C_{3}$.
\end{lem}

\begin{proof}
First assume $G_0 = {\rm L}_{n}^{\e}(q)$, so $H$ is of type ${\rm GL}_{n/k}^{\e}(q^k)$ for some prime $k$ (note that $k$ is odd if $\e=-$). If $n=k$ then $H_0 = Z_a.Z_k$ for some $a \geqs 1$ (see \cite[Proposition 4.3.6]{KL}) and thus $d(M_0) \leqs 2$. On the other hand, if $n>k$ then $H^{\infty}$ is quasisimple and irreducible, so Lemma \ref{lem:sr} implies that $d(M) \leqs 9$.

Next suppose $G_0 = {\rm PSp}_{n}(q)$. If $H$ is of type ${\rm Sp}_{n/k}(q^k)$, or if $n \geqs 6$ and $H$ is of type ${\rm GU}_{n/2}(q)$, then the result follows from Lemma \ref{lem:sr}. Now assume $n=4$ and $H$ is of type ${\rm GU}_{2}(q)$, so $q \geqs 5$ is odd (see \cite[Table 8.12]{BHR}). Here
$H^{\infty}$ is reducible (see \cite[Lemma 4.3.2]{KL}) so we need to argue differently.  According to \cite[Proposition 4.3.7]{KL} we have
$$H_0 = Z_{(q+1)/2}.({\rm PGU}_{2}(q) \times Z_2).$$
In general, $H = N.A$ where $N = Z_{(q+1)/2}$ or $Z_{q+1}$, $A/Z(A)$ has socle ${\rm L}_{2}(q)$ and $Z(A) \leqs Z_2$. If $M$ contains $N$ then $M/N$ is a maximal subgroup of $H/N \cong A$ and we deduce that $d(M) \leqs 8$ since every maximal subgroup of $A/Z(A)$ is $6$-generator by Theorem \ref{t:bls}. On the other hand, if $N \not\leqs M$ then $M = (M\cap N).A$ and $d(M) \leqs d(M \cap N)+d(A) \leqs 5$.

Finally, suppose $G_0 = {\rm P\O}_{n}^{\e}(q)$. If $n \equiv 2 \imod{4}$ and $H$ is of type $O_{n/2}(q^2)$ (with $q$ odd) then $H^{\infty}$ is quasisimple and irreducible, so the result follows from Lemma \ref{lem:sr}. The same argument applies if $n$ is even and $H$ is of type ${\rm GU}_{n/2}(q)$. Finally, let us assume that $H$ is of type $O_{n/k}^{\e}(q^k)$, where $k$ is a prime and $n/k \geqs 3$. By applying Lemma \ref{lem:sr} we reduce to the case where $H$ is of type $O_{4}^{+}(q^k)$, so  $\e=+$ and
$$H_0 = {\rm P\O}_{4}^{+}(q^k).[\ell] = ({\rm L}_{2}(q^k) \times {\rm L}_{2}(q^k)).[\ell] = N.[\ell],$$
where $\ell = (1+\delta_{2,k})k$ (see \cite[Proposition 4.3.14]{KL}). Then $N < H/Z(H) \leqs {\rm Aut}(N)$ and $Z(H) \leqs Z_2$. If $M$ contains $Z(H)$ then Lemma \ref{l:o4} implies that $d(M/Z(H)) \leqs 8$ and thus $d(M) \leqs 9$. Otherwise $M \cong H/Z(H)$ and $d(M) \leqs 6$ by Lemma \ref{l:o4}.
\end{proof}

\begin{lem}\label{l:c4}
Proposition $\ref{cla}$ holds if $H \in \C_{4}$.
\end{lem}

\begin{proof}
First assume $G_0 = {\rm L}_{n}^{\e}(q)$ and $H$ is of type ${\rm GL}_{a}^{\e}(q) \otimes {\rm GL}_{b}^{\e}(q)$, where $n=ab$ and $2 \leqs a <b$. By \cite[Proposition 4.4.10]{KL} we have $H=N.A$ where $N = {\rm L}_{a}^{\e}(q) \times {\rm L}_{b}^{\e}(q)$ and
$A \leqs (Z_{(a,q-\e)} \times Z_{(b,q-\e)}).(Z_f \times Z_2)$. Since $d(N)=2$ we deduce that $d(M) \leqs 6$ if $M$ contains $N$, so assume otherwise. Then $M = (M \cap N).A$ and $d(A) \leqs 4$, so it suffices to show that $d(M \cap N) \leqs 8$. Now $M \cap N$ is a maximal $A$-invariant subgroup of $N$, so
$M \cap N = C \times {\rm L}_{b}^{\e}(q)$ or ${\rm L}_{a}^{\e}(q) \times D$, where
$C = E \cap {\rm L}_{a}^{\e}(q)$ for some maximal subgroup $E$ of a group $F$ with ${\rm L}_{a}^{\e}(q) \leqs F \leqs {\rm Aut}({\rm L}_{a}^{\e}(q))$, and similarly for $D$. By applying  Theorem \ref{t:bls} and Lemma \ref{l:l2} we deduce that $C$ and $D$ are $4$-generator, so $d(M \cap N) \leqs 6$.

Next suppose $G_0 = {\rm PSp}_{n}(q)$ and $H$ is of type ${\rm Sp}_{a}(q) \otimes O_{b}^{\e}(q)$, where $n=ab$, $b \geqs 3$ and $q$ is odd.
Here $H=N.A$, where $N={\rm PSp}_{a}(q) \times {\rm P\O}_{b}^{\e}(q)$ and $A \leqs [2^3].(Z_f \times Z_2)$. In particular, $d(M) \leqs 9$ if $M$ contains $N$. Otherwise $M = (M \cap N).A$, where $M \cap N$ is a maximal $A$-invariant subgroup of $N$, and it suffices to show that $d(M \cap N) \leqs 7$. If both factors of $N$ are simple then
$M \cap N = C \times {\rm P\O}_{b}^{\e}(q)$ or ${\rm PSp}_{a}(q) \times D$, where $C=E \cap {\rm PSp}_{a}(q)$ and $E$ is maximal in an almost simple group with socle ${\rm PSp}_{a}(q)$, and similarly for $D$. By applying Theorem \ref{t:bls} we deduce that $d(M \cap N) \leqs 6$. A very similar argument applies if $(a,q)=(2,3)$, $(b,q)=(3,3)$ or $(b,\e)=(4,+)$, using Lemmas \ref{l:o4} and \ref{l:l2}.

Finally, let us assume $G_0 = {\rm P\O}_{n}^{\e}(q)$. The usual argument applies if $H$ is of type ${\rm Sp}_{a}(q) \otimes {\rm Sp}_{b}(q)$, so let us take $H$ to be of type $O_{a}^{\e_1}(q) \otimes O_{b}^{\e_2}(q)$. Here $q$ is odd, $a,b \geqs 3$ and $(a,\e_1) \ne (b,\e_2)$. For brevity, we will assume that $\e_1=\e_2=+$, so $\e=+$ and $4 \leqs a <b$ (the other cases are very similar). By \cite[Proposition 4.4.14]{KL} we have $H=N.A$, where $N = {\rm P\O}_{a}^{+}(q) \times {\rm P\O}_{b}^{+}(q)$ and $A \leqs (D_8 \times D_8).Z_f$. Note that $d(N) \leqs 4$ and every subgroup of $(D_8 \times D_8).Z_f$ is $5$-generator. In particular, if $M$ contains $N$ then $d(M) \leqs 9$, so assume otherwise. Then $M= (M \cap N).A$ and the usual argument (using Theorem \ref{t:bls} and Lemma \ref{l:o4}) shows that $d(M \cap N) \leqs 6$, whence $d(M) \leqs 11$.
\end{proof}

\begin{lem}\label{l:c5}
Proposition $\ref{cla}$ holds if $H \in \C_{5}$.
\end{lem}

\begin{proof}
First assume $G_0 = {\rm L}_{n}^{\e}(q)$ and $H$ is of type ${\rm GL}_{n}^{\e}(q_0)$, where $q=q_0^k$ for a prime $k$ (with $k$ odd if $\e=-$). Note that $(n,q_0) \neq (2,2)$ (see \cite[Table 8.1]{BHR}). If $(n,q_0)= (2,3)$ then $H \cong A \times B$, where $A \in \{A_4,S_4\}$ and $B \leqs Z_k$, and we deduce that $d(M) \leqs 3$. The same conclusion holds if $\e=-$ and $(n,q_0) = (3,2)$. In every other case, Lemma \ref{lem:sr} implies that $d(M) \leqs 9$. Similarly, we can apply Lemma \ref{lem:sr} if $G_0$ is symplectic or orthogonal, and also if $G_0 = {\rm U}_{n}(q)$ and $H$ is of type ${\rm Sp}_{n}(q)$.

Finally, let us assume $G_0 = {\rm U}_{n}(q)$ and $H$ is of type $O_{n}^{\e}(q)$ (so $q$ is odd and $n \geqs 3$). In view of Lemma \ref{lem:sr} we may assume that $(n,\e) = (4,+)$ (note that $(n,q) \neq (3,3)$; see \cite[Table 8.5]{BHR}). Here $q \geqs 5$ (see \cite[Table 8.10]{BHR}) and
\begin{equation}\label{eq:c5}
H_0 = {\rm PSO}_{4}^{+}(q).2 = ({\rm L}_{2}(q) \times {\rm L}_{2}(q)).[2^2] = N.[2^2].
\end{equation}
More precisely, $N < H/Z(H) \leqs {\rm Aut}(N)$ with $Z(H) \leqs Z_2$, and the result follows by applying Lemma \ref{l:o4}.
\end{proof}

\begin{lem}\label{l:c6}
Proposition $\ref{cla}$ holds if $H \in \C_{6}$.
\end{lem}

\begin{proof}
First assume $G_0 = {\rm L}_{n}(q)$ and $H$ is of type $r^{1+2m}.{\rm Sp}_{2m}(r)$, where $n=r^m$ and $r$ is an odd prime. If $n=3$ then $q=p \equiv 1 \imod{3}$ (see \cite[Table 8.3]{BHR}), $3^2{:}Q_8 \leqs H \leqs {\rm AGL}_{2}(3)$ and it is easy to check that $d(M) \leqs 3$.
Now assume $n \geqs 5$, in which case
$$H = W{:}({\rm Sp}_{2m}(r).A) \leqs W{:}{\rm GSp}_{2m}(r)$$
and $A \leqs Z_{2f}$, where $W=r^{2m}$ and $q=p^f$, with $f$ an odd divisor of $r-1$ (see \cite[Proposition 4.6.5]{KL}). Since $W$ is the unique minimal normal subgroup of $H$ we may assume that $M = W.J$ for some maximal subgroup $J<{\rm Sp}_{2m}(r).A$, so Lemma \ref{aff} implies that $d(M) \leqs 8$. An entirely similar argument applies if $G_0 = {\rm U}_{n}(q)$. If $G_0 = {\rm L}_{2}(q)$ and $H$ is of type $2^{1+2}_{-}.O_{2}^{-}(2)$ then $H=A_4$ or $S_4$ and the result follows.

Next assume $G_0 = {\rm P\O}_{n}^{+}(q)$ and $H$ is of type $2^{1+2m}_{+}.O_{2m}^{+}(2)$, so $q=p \geqs 3$ and $n=2^m$ with $m \geqs 3$. By \cite[Proposition 4.6.8]{KL} we have $H = W.A$ with $W=2^{2m}$ and $A = \O_{2m}^{+}(2)$ or $O_{2m}^{+}(2)$. In particular, $W$ is the unique minimal normal subgroup of $H$ so we may assume that $M = W.J$ with $J$ maximal in $A$. By applying Lemma \ref{aff} we deduce that $d(M) \leqs 1+d(J) \leqs 7$. The case where $G_0 = {\rm PSp}_{n}(q)$ and $H$ is of type $2^{1+2m}_{-}.O_{2m}^{-}(2)$ is entirely similar.
\end{proof}

\begin{lem}\label{l:c7}
Proposition $\ref{cla}$ holds if $H \in \C_{7}$.
\end{lem}

\begin{proof}
We refer the reader to \cite[Table 4.7.A]{KL} for the list of cases that we need to consider.
First assume $G_0 = {\rm L}_{n}^{\e}(q)$ and $H$ is of type ${\rm GL}_{a}^{\e}(q) \wr S_t$ with $a \geqs 3$. Here $n=a^t$ and $(a,q,\e) \neq (3,2,-)$. We will assume $\e=+$ since the case $\e=-$ is very similar. To begin with, let us assume that at least one of the following three conditions does \emph{not} hold:
\begin{equation}\label{e:cond}
t = 2, \;\; a \equiv 2 \imod{4}, \;\; q \equiv -1 \imod{4}.
\end{equation}
Write ${\rm PGL}_{a}(q) = {\rm L}_{a}(q).\la \delta \ra$ and ${\rm L}_{a}(q) = \la x, y\ra$ where $|x|$ and $|y|$ are coprime. Set $d=(a,q-1)$. According to \cite[Proposition 4.7.3]{KL} we have $H_0 = N_0.S_t$ and $H=N.S_t$, where
$$N_0 = {\rm L}_{a}(q)^t.A_0, \;\; N = {\rm L}_{a}(q)^t.A.2^b.Z_k$$
where $A_0 = \frac{1}{c}(Z_d)^t \leqs \frac{1}{e}(Z_d)^t=A$, $b \in \{0,1\}$ and $k$ divides $\log_pq$, for some divisors $c,e$ of $d$. If $M$ contains $N$ then $M_0 = N_0.J$ for some maximal subgroup $J<S_t$ and the result quickly follows. For example, if $J$ is a transitive subgroup then $M_0$ is generated by $(x,y, 1, \ldots, 1)$, $(\delta, \delta^{-1}, 1, \ldots, 1)$ and $(\delta^{\ell}, 1, \ldots, 1)$ for some $\ell \geqs 0$, together with at most $4$ generators for $J$.

Now assume $N \not\leqs M$, so $M = (M \cap N).S_t$ and $M \cap N$ is a maximal $S_t$-invariant subgroup of $N$. Since $S={\rm L}_{a}(q)^t$ is the unique minimal normal subgroup of $H$, we may assume that $M$ contains $S$. Set $\bar{M}=M/S$ and $\bar{H}=H/S = A.2^b.Z_k.S_t$. If $\bar{M}$ contains $A$ then $M_0 = H_0$ and thus $d(M_0) \leqs 4$, so assume otherwise. Then $\bar{M} = (\bar{M} \cap A).2^b.Z_k.S_t$ and $(\bar{M} \cap A).S_t< A.S_t$ is maximal, so Lemma \ref{l:wreath} implies that $d((\bar{M} \cap A).S_t) \leqs 6$ and we deduce that $d(\bar{M}) \leqs 8$ and $d(M) \leqs 10$.

To complete the analysis of this case, we may assume that all of the conditions in \eqref{e:cond} are satisfied. The above argument goes through unchanged if $H$ contains an element that interchanges the two copies of ${\rm L}_{a}(q)$ in the socle of $H$, so we may assume that
$$H=(N_1 \times N_2).A.2^b.Z_k,$$
where $N_i = {\rm L}_{a}(q)$ and $A,b$ and $k$ are as above. Note that $N_1$ and $N_2$ are the minimal normal subgroups of $H$. If $M$ contains both of these subgroups then the previous argument goes through, so we may assume that $M$ contains $N_1$ but not $N_2$. Set $\bar{M} = M/N_1$ and $\bar{H} = H/N_1 = N_2.A.2^b.Z_k$. Since $N_2 \not\leqs \bar{M}$ we have $\bar{M} = (\bar{M} \cap N_2).A.2^b.Z_k$ and $\bar{M} \cap N_2 = {\rm L}_{a}(q) \cap B$ where $B$ is a maximal subgroup of an almost simple group with socle ${\rm L}_{a}(q)$. By Theorem \ref{t:bls} we have $d(\bar{M} \cap N_2) \leqs 4$, so $d(\bar{M}) \leqs 8$ and thus $d(M) \leqs 10$.

\vs

The remaining $\mathcal{C}_7$ cases are similar, so we only give details in the situation where $G_0 = {\rm P\O}_{n}^{+}(q)$ and $H$ is of type $O_{a}^{+}(q) \wr S_t$, with $a \geqs 6$ and $q$ odd. Let $D$ and $D'$ be the discriminants of the quadratic forms corresponding to $O_{n}^{+}(q)$ and $O_{a}^{+}(q)$, respectively. Note that $n(q-1)/4$ is always even, so $D=\square$ (see \cite[Proposition 2.5.10]{KL}). Write ${\rm PO}_{a}^{+}(q) = {\rm PSO}_{a}^{+}(q).\la r \ra$ and ${\rm PGO}_{a}^{+}(q) = {\rm PO}_{a}^{+}(q).\la \delta \ra$ for involutions $r$ and $\delta$. Also fix $x,y \in {\rm PSO}_{a}^{+}(q)$ such that ${\rm PSO}_{a}^{+}(q) = \la x,y \ra$. Two cases require special attention:
\begin{itemize}\addtolength{\itemsep}{0.2\baselineskip}
\item[(a)] $t=2$ and $a \equiv 2 \imod{4}$;
\item[(b)] $t=3$, $a \equiv 2 \imod{4}$ and $q \equiv 3 \imod{4}$.
\end{itemize}

For now, we will assume that we are not in one of these cases. By \cite[Proposition 4.7.6]{KL} we have $H_0 = N_0.S_t$ and $H=N.S_t$, where
$$N_0 = {\rm PSO}_{a}^{+}(q)^t.[2^{2t-1}],\;\; N = {\rm PSO}_{a}^{+}(q)^t.[2^{i}].Z_k$$
with $i \in \{2t-1,2t\}$ and $k$ a divisor of $\log_pq$. If $M$ contains $N$ then $M_0 = N_0.J$ with $J<S_t$ maximal and the result quickly follows. For example, if $J$ is transitive then $M_0$ is generated by $(x, 1, \ldots, 1)$, $(y, 1, \ldots, 1)$, $(r,1,\ldots, 1)$ and $(\delta, \delta^{-1}, 1, \ldots, 1)$, together with at most $4$ generators for $J$.

Now assume $M = (M \cap N).S_t$. First consider the case where $D'=\square$, so ${\rm PSO}_{a}^{+}(q) = {\rm P\O}_{a}^{+}(q).2$ and $S={\rm P\O}_{a}^{+}(q)^t$ is the unique minimal normal subgroup of $H$. As usual, we may assume that $M$ contains $S$, so set $\bar{M} = M/S$ and $\bar{H}=H/S=A.Z_k.S_t$, where $A=[2^{t+i}]$. If $\bar{M}$ contains $A$ then $M_0 = H_0$ and thus $d(M_0) \leqs 4$. Otherwise, $\bar{M} = (\bar{M} \cap A).Z_k.S_t$ and $(\bar{M} \cap A).S_t$ is a maximal subgroup of $A.S_t = \frac{1}{b}(D_8 \wr S_t)$, where $b=1$ or $2$. By Lemma \ref{l:wreath} we have $d((\bar{M} \cap A).S_t) \leqs 6$, so $d(\bar{M}) \leqs 7$ and thus $d(M) \leqs 9$. A similar argument applies if $D'=\boxtimes$. Here ${\rm PSO}_{a}^{+}(q) = {\rm P\O}_{a}^{+}(q)$ and once again we may assume that $M$ contains $S = {\rm P\O}_{a}^{+}(q)^t$. The rest of the argument goes through, replacing $D_8$ by $V_4$.

It remains to handle the cases described in (a) and (b) above. First consider (a). We will assume $D'=\square$ (the other case is very similar), so $H = {\rm P\O}_{a}^{+}(q)^2.[2^{b+2}].Z_k.Z_c$, where $b \in \{2,3,4\}$, $c \in \{1,2\}$ and $k$ divides $\log_pq$. Note that $[2^{b+2}] \leqs D_8 \times D_8$ is $4$-generator. If $c=2$ then the previous argument goes through, so let us assume $c=1$. Here $H$ has two minimal normal subgroups $N_1$ and $N_2$, both isomorphic to ${\rm P\O}_{a}^{+}(q)$. If $M$ contains $S=N_1 \times N_2$ then $M/S<H/S=[2^{b+2}].Z_k$, so $d(M/S) \leqs 5$ and thus $d(M) \leqs 7$. Therefore, we may assume that $H$ contains $N_1$ but not $N_2$. Set $\bar{M}=M/N_1$ and $\bar{H}=H/N_1 = N_2.[2^{b+2}].Z_k$. Then $\bar{M} = (\bar{M} \cap N_2).[2^{b+2}].Z_k$ and Theorem \ref{t:bls} implies that $d(\bar{M} \cap N_2) \leqs 4$, so $d(\bar{M}) \leqs 9$ and thus $d(M) \leqs 11$.

Finally, let us assume that the conditions in (b) hold, so $D'=\boxtimes$ and $H = S.A.Z_k.B$, where $S={\rm P\O}_{a}^{+}(q)^3$, $A=[2^b]$ with $b \in \{5,6\}$, $k$ divides $\log_pq$ and $B \in \{Z_3, S_3\}$. In the usual way, it is easy to reduce to the case where $M = (M \cap N).B$.
Now $B$ acts transitively on the factors of $S$, so $S$ is the unique minimal normal subgroup of $H$ and we may assume that $M$ contains $S$. Set $\bar{M}=M/S$ and $\bar{H}=H/S = A.Z_k.B$. If $\bar{M}$ contains $A$ then $M_0=H_0$ and $d(M_0) \leqs 4$, so assume otherwise. Then $\bar{M} = (\bar{M} \cap A).Z_k.B$ and $(\bar{M} \cap A).B$ is a maximal subgroup of $A.B = \frac{1}{c}(V_4 \wr B)$, where $c \in \{1,2\}$. Therefore $d((\bar{M} \cap A).B) \leqs 6$ by Lemma \ref{l:wreath}, so $d(\bar{M}) \leqs 7$ and thus $d(M) \leqs 9$ since $S$ is $2$-generator.
\end{proof}

\begin{lem}\label{l:c8}
Proposition $\ref{cla}$ holds if $H \in \C_{8}$.
\end{lem}

\begin{proof}
First assume $G_0 = {\rm L}_{n}(q)$. If $H$ is of type ${\rm Sp}_{n}(q)$, then $n \geqs 4$ and Lemma \ref{lem:sr} applies. Next suppose $H$ is of type $O_{n}^{\e}(q)$. If $(n,q) \neq (3,3)$ and $(n,\e) \neq (4,+)$, then we can use Lemma \ref{lem:sr} once again. It is easy to check that $d(M) \leqs 3$ if $(n,q) = (3,3)$. If $(n,\e) = (4,+)$ then \eqref{eq:c5} holds and we can repeat the argument in the proof of Lemma \ref{l:c5}. Finally, suppose that $H$ is of type ${\rm U}_{n}(q_0)$, where $n \geqs 3$ and $q=q_0^2$. If $(n,q) = (3,4)$ then $d(M) \leqs 3$, otherwise the result follows from Lemma \ref{lem:sr}.

Finally let us assume that $G_0 = {\rm PSp}_{n}(q)$ and $H$ is of type $O_{n}^{\e}(q)$, where $q$ is even, $n \geqs 4$ and $(n,q) \neq (4,2)$. If $(n,\e) \neq (4,+)$ then $H$ is almost simple and thus $d(M) \leqs 6$. On the other hand, if $(n,\e) = (4,+)$ then
$$H_0 = O_{4}^{+}(q) = {\rm L}_{2}(q) \wr S_2 = ({\rm L}_{2}(q) \times {\rm L}_{2}(q)).2 = N.2$$
and $N < H \leqs  {\rm Aut}(N)$, so Lemma \ref{l:o4} implies that $d(M) \leqs 8$.
\end{proof}

To complete the proof of Proposition \ref{cla}, it remains to deal with certain \emph{novelty} subgroups $H$ of $G$, where $H_0=H \cap G_0$ is non-maximal in $G_0$. In view of  \cite{asch} and our earlier work, we may assume that one of the following holds:

\begin{itemize}\addtolength{\itemsep}{0.2\baselineskip}
\item[(a)] $G_{0}={\rm PSp}_{4}(q)$, $q$ even and $G$ contains a graph-field  automorphism;
\item[(b)] $G_{0}={\rm P\O}_{8}^{+}(q)$ and $G$ contains a triality automorphism.
\end{itemize}

In \cite[Section 14]{asch}, Aschbacher proves a version of his main theorem which describes the various possibilities for $H$ in case (a), but his theorem does not apply in case (b); here
the possibilities were determined later by Kleidman \cite{K}. We record the relevant non-parabolic subgroups in Table \ref{t:n}. Note that in case (a) we may assume $q>2$ since ${\rm PSp}_{4}(2)'\cong A_{6}$.

\begin{table}[h]
$$\begin{array}{lll} \hline\hline
G_0 &  \mbox{Type of $H$} &  \mbox{Conditions} \\ \hline
{\rm PSp}_{4}(q) & O_{2}^{\e}(q) \wr S_{2} &  \mbox{$q>2$ even} \\
& O_{2}^{-}(q^{2}).2  &  \mbox{$q>2$ even} \\
{\rm P\O}_{8}^{+}(q) & {\rm GL}_{3}^{\e}(q) \times {\rm GL}_{1}^{\e}(q)  & \\
 & O_{2}^{-}(q^{2})\times O_{2}^{-}(q^{2}) & \\
 & [2^9].{\rm SL}_{3}(2) & q=p>2 \\ \hline\hline
\end{array}$$
\caption{Some novelty subgroups}
\label{t:n}
\end{table}

\begin{lem}\label{l:n1}
Proposition \ref{cla} holds if $G_0 = {\rm PSp}_{4}(q)$ and $H$ is in Table \ref{t:n}.
\end{lem}

\begin{proof}
Here $H_0 = D_{2(q\pm 1)} \wr S_2$ or $Z_{q^2+1}.4$, so $d(M_0) \leqs 4$ and the result follows.
\end{proof}

\begin{lem}\label{l:n2}
Proposition \ref{cla} holds if $G_0 = {\rm P\O}_{8}^{+}(q)$ and $H$ is in Table \ref{t:n}.
\end{lem}

\begin{proof}
As before, it suffices to show that $d(M_0) \leqs 9$. First assume $H$ is of type ${\rm GL}_{3}^{\e}(q) \times {\rm GL}_{1}^{\e}(q)$. Set $d=(2,q-1)$. Working in $\O_8^{+}(q)$ we have $H_0 = N_0.Z_{(q-\e)/d}.[2^2]$ and $H = N_0.A$, where $N_0 = \frac{1}{d}{\rm GL}^{\e}_{3}(q)$ and $A=Z_{(q-\e)/d}.[2^a].B.Z_k$ with $a \in \{2,3,4\}$, $B \in \{Z_3, S_3\}$ and $k$ a divisor of $\log_pq$. If $M$ contains $N_0$ then $M_0 = N_0.C$ and $C \leqs Z_{(q-\e)/d}.[2^2]$ is $3$-generator, so $d(M_0) \leqs 5$. Now assume $N_0 \not\leqs M$, so $M = (M \cap N_0).A$ and $M_0 = (M \cap N_0).Z_{(q-\e)/d}.[2^2]$, where $M \cap N_0$ is a maximal $A$-invariant subgroup of $N_0$. Now $M \cap {\rm SL}^{\e}_{3}(q) = D \cap {\rm SL}^{\e}_{3}(q)$, where $D$ is maximal in a group $E$ of the form
$${\rm SL}_{3}^{\e}(q) \leqs E \leqs {\rm \Gamma L}_{3}^{\e}(q).\la \gamma \ra,$$
where $\gamma$ is a graph automorphism. By applying Theorem \ref{t:bls} we deduce that $d(M \cap {\rm SL}_{3}^{\e}(q)) \leqs 5$, so $d(M \cap N_0) \leqs 6$ and thus $d(M_0) \leqs 9$ as required.

If $H$ is of type $O_{2}^{-}(q^{2})\times O_{2}^{-}(q^{2})$ then $H_0 = (D_{2l} \times D_{2l}).2^2$, where $l=(q^2+1)/(2,q-1)$ is odd, and we deduce that $d(M_0) \leqs 5$ since every subgroup of $D_{2l} \times D_{2l}$ is $3$-generator. In the final case we have $H_0 = [2^9].{\rm SL}_{3}(2)$ and using {\sc Magma} one can check that every subgroup of $H_0$ is $8$-generator. In particular, $d(M_0) \leqs 8$ and the result follows.
\end{proof}

\vs

This completes the proof of Proposition \ref{cla}.

\section{Exceptional groups}\label{s:excep}

In this section we turn to the exceptional groups of Lie type, establishing Theorem \ref{max} for the second maximal subgroups lying in a maximal non-parabolic subgroup.

\begin{prop}\label{ex:main}
Suppose $M<H<G$ with each subgroup maximal in the next, where $G$ is an almost simple exceptional group of Lie type and $H$ is  non-parabolic. Then $d(M) \leqs 12$.
\end{prop}

\begin{proof}
Let $G_0$ be the socle of $G$, and $H_0 = H\cap G_0$, $M_0 = M\cap G_0$. Write $G_0 = G(q)$, an exceptional simple group of Lie type over $\F_q$, where $q=p^e$, $p$ prime. With the aid of {\sc Magma}, it is easy to check that $d(M) \leqs 4$ if $G_0 = \,^2\!F_4(2)',\,G_2(3)$ or $^3\!D_4(2)$, so we may assume otherwise. As $d(G/G_0) \leqs 2$ it is sufficient to show that $d(M_0) \leqs 10$.

According to \cite[Theorem 2]{LS10}, the possibilities for $H_0$ are as follows:
\begin{itemize}\addtolength{\itemsep}{0.2\baselineskip}
\item[(i)] $H_0$ is almost simple;
\item[(ii)] $H_0 = N_{G_0}(K)$, where $K$ is a reductive subgroup of $G_0$ of maximal rank, not a maximal torus; the
possibilities are listed in \cite[Table 5.1]{LSS};
\item[(iii)] $H_0= N_{G_0}(T)$, where $T$ is a maximal torus
of $G_0$; the possibilities are listed in \cite[Table 5.2]{LSS};
\item[(iv)] The generalized Fitting subgroup $F^*(H_0)$ is as in \cite[Table III]{LS10};
\item[(v)] $H_0 = N_{G_0}(E)$, where $E$ is an elementary abelian group
given in \cite[Theorem 1(II)]{CLSS}.
\end{itemize}

In case (i), $d(M_0) \leqs 4$ by Theorem \ref{t:bls}.

In case (iv), with two exceptions $H_0$ has a subgroup $H_1$ of index at most 6 that is a direct product $S_1\times S_2$ of non-isomorphic simple groups $S_i$; in the exceptional cases, $H_0$ has a subgroup $H_1\cong {\rm L}_2(q)^2$ or ${\rm L}_2(q) \times G_2(q)^2$ of index dividing 4. Excluding the exceptional cases, we must have $M_0 \cap H_1 = S_1 \times M_2$ where either $M_2=S_2$ or $M_2$ is a maximal $H$-invariant subgroup of $S_2$. Using Theorem \ref{t:bls} we see that $d(M_2) \leqs 4$, so $d(M_0) \leqs 8$. The first exceptional case $H_1 = S_1\times S_2 \cong {\rm L}_2(q)^2$ is entirely similar: either $M_0\cap H_1 = S_1 \times M_2$ as above, or it is a diagonal subgroup isomorphic to ${\rm L}_2(q)$.
In the second exceptional case, the two $G_2(q)$ factors are interchanged by an element of $H_0$, so
either $M_0\cap H_1$ is $M_1 \times G_2(q)^2$ with $M_1$ maximal $H$-invariant in ${\rm L}_2(q)$, or it is
 ${\rm L}_2(q) \times D$ where $D$ is a diagonal subgroup of $G_2(q)^2$ isomorphic to $G_2(q)$. In every case we easily see that $d(M_0) \leqs 8$ using Theorem \ref{t:bls}.

Next consider case (v). In this case, either $H_0$ is one of the groups
\begin{equation}\label{loc}
5^3.{\rm SL}_3(5), \,2^{5+10}.{\rm SL}_5(2), \,3^{3+3}.{\rm SL}_3(3), \,3^3.{\rm SL}_3(3), \,2^3.7, \,2^3.{\rm SL}_3(2),
\end{equation}
 or $G_0 = E_7(q)$ and $H_0 = (2^2\times {\rm P\O}_8^+(q)).S_3$ with $q$ odd.
In the latter case, either $M_0$ contains ${\rm P\O}_8^+(q)$ in which case $d(M_0) \leqs 4$, or $M_0 \cap {\rm P\O}_8^+(q)$ is a maximal $S_3$-invariant subgroup, in which case $d(M_0) \leqs 6$ by Theorem \ref{t:bls}. The only problematic possibility in the list \eqref{loc} is
$H_0 = 2^{5+10}.{\rm SL}_5(2)$. Let $P = O_2(H_0)$. Then $\Phi(P) = 2^5$ and $H_0/P \cong {\rm SL}_5(2)$ acts on $P/\Phi(P)$ as the wedge-square of the natural module. If $M_0$ contains $P$, then $M_0 = P.X$ where $X$ is maximal in ${\rm SL}_5(2)$; by inspecting \cite[Tables 8.18, 8.19]{BHR} we see that $X$ is either a parabolic subgroup or $31{:}5$, and so has at most 3 composition factors on $P/\Phi(P)$. In particular, we deduce that $d(M_0) \leqs 3+d(X) \leqs 7$ in this case. And if $P\not \leqs M_0$ then $M_0 = \Phi(P).{\rm SL}_5(2)$ and hence $d(M_0) \leqs 3$.

Next we handle case (iii). Here $H_0= N_{G_0}(T)$, where $T$ is a maximal torus
of $G_0$, as listed in \cite[Table 5.2]{LSS}. The groups $W = N_{G_0}(T)/T$ are also listed in Table 5.2 of \cite{LSS}; these are subgroups of the Weyl group of $G_0$.

Suppose first that $T \leqs M_0$, so that $M = T.X$ with $X$ maximal in $N_G(T)/T$ (which is $W \times \langle \phi\rangle$, possibly extended by a graph automorphism, where $\phi$ is a field automorphism). If $T \ne (q\pm 1)^r$ with $r \in \{7,8\}$, it is clear from the list that $d(T) \leqs 6$, and one checks that $d(X) \leqs 6$ also, giving $d(M) \leqs 12$. And if $T = (q\pm 1)^r$ then $W = W(E_r)$ and one checks that $d(X) \leqs 4$ for a maximal subgroup in this case, giving $d(M) \leqs r+4 \leqs 12$.

Now suppose $T \not \leqs M_0$. Then $M_0 = (M\cap T).W$. A check gives $d(W) \leqs 2$, hence $d(M_0) \leqs d(M\cap T)+2 \leqs 10$.

It remains to handle case (ii), in which $H_0 = N_{G_0}(K)$, where $K$ is a reductive subgroup of $G_0$ of maximal rank, not a maximal torus. The possibilities for $K$ and $H_0/K$ are listed in \cite[Table 5.1]{LSS}. In all cases $K$ is a central product $\prod L_i \circ R$, where each $L_i$ is either quasisimple or in $\{{\rm SL}_2(2),{\rm SL}_2(3),{\rm SU}_3(2)\}$, and $R$ is an abelian $p'$-group of rank at most 2 (also $R=1$ unless $G_0$ is of type $E_7, E_6^\e$ or $^3\!D_4$). 

The cases where $K$ is solvable are those in Table \ref{solc}. We exclude these cases from consideration until the end of the proof.

\begin{table}
$$\begin{array}{llc} \hline\hline
G_0 &  K & q \\ \hline
E_8(q) & A_1(q)^8 & 2,3  \\
E_8(q) & A_2^-(q)^4 & 2 \\
E_7(q) & A_1(q)^7 & 2,3  \\
{}^2E_6(q) & A_2^-(q)^3 & 2  \\
F_4(q) & A_2^-(q)^2 & 2\\ \hline\hline
\end{array}$$
\caption{Cases with $K$ solvable}
\label{solc}
\end{table}

Let $N = {\rm core}_H(M)$. By Lemma \ref{lem1} we may assume that $N\ne 1$. Assume first that $K\leqs N$. Then $M = K.X$ where $X$ is maximal in $H/K$. Inspecting the list of possibilities for $K$ and $H/K$, it is easy to check that $d(K) \leqs 4$ and $d(X) \leqs 8$, giving the conclusion.

Next assume that $N \leqs Z(K)$. Then $H=MK$ so $d_M(N) = d_H(N)$. Inspection of the list shows that $d(N) \leqs 2$ except for the cases $K = A_1(q)^r$ ($r=7,8$), and in these cases $Z(K) = 2^{r-4}$ and $d_M(N) \leqs 2$. Hence by Remark \ref{clev} we have
$d(M) \leqs d_M(N)+10 \leqs 12$, as required.

Now assume $K \not \leqs N$ and $N \not \leqs Z(K)$. Then $N$ contains a product $N_0$ of factors $L_i$ of $K$. In all but two cases in the list where $K$ has at least two isomorphic factors $L_i$, $H_0/K$ acts transitively on these factors; the two exceptional cases are $K = A_2^\e(q)^2$ in $F_4(q)$ and $K = A_1(q)^2$ in $G_2(q)$. Hence inspecting the list, we see that $K$ is in Table \ref{tab:kk}, with $N_0$ equal to one of the factors (or $A_1(q)^3$):

\begin{table}
$$\begin{array}{ll} \hline\hline
G_0 & K \\ \hline
E_8(q) & A_1(q)E_7(q),\;A_2^\e(q)E_6^\e(q) \\
E_7(q) & A_1(q)D_6(q),\;A_2^\e(q)A_5^\e(q),\;^3\!D_4(q)A_1(q^3),\;D_4(q)A_1(q)^3,\;E_6^\e(q)\circ (q-\e) \\
E_6^\e(q) & A_1(q)A_5^\e(q),\;A_2(q^2)A_2^{-\e}(q),\;^3\!D_4(q)\times (q^2+\e q+1),\; \\
  & D_5^\e(q)\circ (q-\e),\;D_4(q)\circ (q-\e)^2  \\
F_4(q) & A_1(q)C_3(q),\; A_2^\e(q)^2 \\
G_2(q) & A_1(q)^2 \\
^3\!D_4(q) & A_1(q)A_1(q^3),\; A_2^\e(q)\circ (q^2+\e q+1) \\ \hline\hline
\end{array}$$
\caption{Cases with $K \not\leqs N$ and $N \not\leqs Z(K)$}
\label{tab:kk}
\end{table}

Write $K = N_0K_0$, where $K_0$ is the product of the factors $L_i$ (or $R$) not in $N_0$. Then $M\cap K = N_0M_0$, where $M_0$ is a maximal $H$-invariant subgroup of $K_0$. From the above table, $K_0$ is either a single factor $L_i$ or $R$  of $K$, or it is $A_1(q)^3$. In the former case, using Theorem \ref{t:bls} we see that $d(M_0) \leqs 4$, whence $d(M) \leqs d(N_0)+d(M_0)+d(H/K) \leqs 12$. The other possibility is that $K_0 = A_1(q)^3$, $N_0 = D_4(q)$. If $q>3$ then $M_0$ must be a diagonal subgroup of $K_0$, so $d(M_0)\leqs 2$; and if $q\leqs 3$ then $H/N_0 \cong A_1(q)^3.d^3.S_3$ where $d=(2,q-1)$, and we easily check that $d(M/N_0) \leqs 10$, so that $d(M) \leqs d(N_0) + 10\leqs 12$.

It remains to handle the cases where $K$ is solvable, given in Table \ref{solc}. The most complicated example is $K = A_1(q)^8$ in $E_8(q)$ with $q=3$. We deal with this case and leave the others to the reader. In this case $Z(K) = 2^4$, $H/K \cong 2^4.{\rm AGL}_3(2)$, so
\[
H = 2^4.2^{16}.3^8.2^4.2^3.{\rm L}_3(2).
\]
Let $R$ denote the solvable radical of $H$. If $R\leqs M$ then $M = R.X$ where $X$ is maximal in ${\rm L}_3(2)$; since $d(X)=2$ and $d_M(A_1(3)^8) = 2$, it follows that $d(M) \leqs 2 + d(2^4.2^3) + d(X) < 12$. And if $R \not \leqs M$ then $M/M\cap R \cong {\rm L}_3(2)$ and it follows that $d_M(M\cap R) \leqs 10$, whence $d(M) \leqs 10+d({\rm L}_3(2)) = 12$.
\end{proof}

\section{Parabolic subgroups and Number Theory}\label{s:parab}

In this section we complete the proof of Theorems \ref{max} and \ref{open} by handling second maximal subgroups 
$M$ lying in parabolic subgroups. In particular we relate the boundedness of $d(M)$ to the number-theoretic
question \eqref{e:star} stated in the Introduction.

\begin{lem}\label{bbdfield}
Let $q = p^k$, where $p$ is a prime and $k \geqs 1$, let $e$ be a divisor of $q-1$ and let $E$ be the subgroup of order $e$ of
the multiplicative group $\F_q^{\times}$. Let $M = \F_q.E$ be the
corresponding subgroup of the semidirect product $\F_q. \F_q^{\times}\cong {\rm AGL}_1(q)$. Then
\[
k/\ell \leqs d(M) \leqs k/\ell + 1,
\]
where $\ell = \min \{ i \geqs 1\,: \,\mbox{$e$ divides $p^i-1$} \}$ is the multiplicative order of $p$ modulo $e$.
\end{lem}

\begin{proof}
Let $K$ be the minimal subfield of $\F_q$ containing $E$. Then $K$ has order $p^{\ell}$
where $\ell$ divides $k$. Therefore $\F_q$ has dimension $k/\ell$ as a vector space
over $K$. Thus $M$ is generated by a basis of that vector space together with a
generator of the cyclic group $E$, so $d(M) \leqs k/\ell + 1$.

To prove the other inequality, suppose $(a_i,b_i)$ are generators for $M$, where
$a_i \in \F_q$, $b_i \in E$ and $i = 1, \ldots, d$. Then $a_1, \ldots, a_d$ generate $\F_q$ as a vector space over $K$, so $d \geqs k/\ell$, as required.
\end{proof}

The next result helps in establishing a connection between bounding the number of generators of second maximal
subgroups and the answer to the number-theoretic question \eqref{e:star} stated in Section \ref{s:intro}.

\begin{lem}\label{arb}
Let $G = {\rm L}_2(q)$, ${}^2B_2(q)$ or ${}^2G_2(q)$ where $q=p^k$ ($p$ prime), let $d = (2,q-1)$, $1$ or $1$ respectively,
and let $B = UT$ be a Borel subgroup of $G$ with unipotent normal subgroup $U$ and Cartan subgroup $T$ of index $d$ in  $\F_q^{\times}$. Let $s$ be a prime divisor of $q-1$ and let $e=\frac{q-1}{ds}$, so that $B$ has a maximal subgroup $M = U.e$ of index $s$. Let $\ell$ be the multiplicative order of $p$ modulo $e$. Then
\begin{itemize}\addtolength{\itemsep}{0.2\baselineskip}
\item[{\rm (i)}] we have $k/\ell \leqs d(M) \leqs  k/\ell +1$;
\item[{\rm (ii)}] $d(M)$ is unbounded if and only if $\ell = o(k)$;
\item[{\rm (iii)}] either $k \in \{\ell,2\ell\}$ (in which case $d(M) \leqs 3$), or $\frac{p^k-1}{p^{\ell}-1} = s$ is prime.
\end{itemize}
\end{lem}

\begin{proof}

We first prove part (i).
If $G = {\rm L}_2(q)$, then $U \cong \F_q$ and so $k/\ell \leqs d(M) \leqs k/\ell +1$ by the previous lemma.
The other families ${}^2B_2(q)$ and ${}^2G_2(q)$ are handled by the same argument, noting that $U/\Phi(U) \cong \F_q$ with $T$ acting by scalar multiplication (see \cite{Suz, Ward}).

Part (ii) follows immediately from part (i). To prove (iii), note that $ds = \frac{p^k-1}{p^{\ell}-1} \cdot \frac{p^{\ell}-1}{e}$ and $s$ is a prime. If $\frac{p^k-1}{p^{\ell}-1} \ne 1,s$, then $d=2$ and $2s = \frac{p^k-1}{p^{\ell}-1}$. This implies that $\frac{k}{\ell}$ is even, say $\frac{k}{\ell} = 2m$. Then writing $q_0 = p^{\ell}$, we have $2s = \frac{(q_0^m-1)(q_0^m+1)}{q_0-1}$, which forces $m=1$, hence $k=2\ell$. This proves (iii). \end{proof}

\begin{lem}\label{borel}
Let $G$ be an almost simple group of Lie type with socle $G_0$. Suppose $G$ has a maximal Borel subgroup $B$, and suppose $B$ has a maximal subgroup $M$ with $d(M)>12$. Then $G_0 = {\rm L}_2(q)$, ${}^2B_2(q)$ or ${}^2G_2(q)$, and $M\cap G_0$ is as in Lemma $\ref{arb}$.
\end{lem}

\begin{proof}
These are the cases where $G_0$ has $BN$-rank 1, or is ${\rm L}_3(q)$, $C_2(2^e)$ or $G_2(3^e)$ and $G$ contains a graph or graph-field automorphism. We need to rule out the latter three cases, and also the case where $G_0 = {\rm U}_3(q)$. As before, set $M_0 = M \cap G_0$. Note that if $G_0 = {\rm L}_2(q)$, ${}^2B_2(q)$ or ${}^2G_2(q)$, then Lemma \ref{lem1} shows that $M\cap G_0$ is as in Lemma $\ref{arb}$.

Consider first $G_0 = {\rm U}_3(q)$, so that $B \cap G_0 = QT$ where $Q = q^{1+2}$ is a special group with $Q' = \Phi(Q) \cong \F_q$, $Q/Q' \cong \F_q^2$ and $T \cong Z_{(q^2-1)/d}$ with $d = (3,q+1)$. If $Q \leqs M$ then $M = QS$ where $S$ contains either $Z_{q-1}$ or $Z_{(q+1)/d}$ (note that $d(S) \leqs 2$). Using Lemma \ref{bbdfield} we see that $d_{M/Q'}(Q/Q') \leqs 4$ and it follows that $d(M) \leqs 4+d(S)+1\leqs 7$, a contradiction. And if $Q \not\leqs M$ then $M_0=(M \cap Q).T$ and $M\cap Q$ is a maximal $T$-invariant subgroup of $Q$; it follows that $d_{M_0}(M\cap Q) \leqs 2$, so  $d(M_0) \leqs 2+d(T)=3$ and thus $d(M) \leqs 5$, a contradiction.

Next consider $G_0 = C_2(q)$ where $q = 2^e$ and $G$ contains an element inducing a graph-field automorphism on $G_0$. Adopting the notation of \cite{car}, let $B\cap G_0 = QT$ where $Q$ is generated by the positive root groups relative to a fixed root system (so $|Q|=q^4$), and $T = \langle h_\a(t),h_\b(u) \,:\, t,u \in \F_q \rangle$, where $\a,\b$ are fundamental roots with $\a$ long and $\b$ short. By assumption, $G = G_0\langle \tau\rangle$, where $\tau$ is a graph-field automorphism of $G_0$ normalizing $Q$ and $T$, sending
\[
h_\a(t) \mapsto h_\b(t^r),\;\;h_\b(u) \mapsto h_\a(u^{2r}),
\]
where $r=2^f$ for some $f \leqs e$. Let $\pi_1,\pi_2 : T \rightarrow \F_q^{\times}$ be the maps sending $h_\a(t), h_\b(u)$ to $t, u$ respectively.

Assume first that $Q \leqs M$, so $M_0 = QT_0$ and $T_0$ is a maximal $\tau$-invariant subgroup of $T$.

If $\pi_1(T_0) = \F_q^{\times}$ then $\pi_2(T_0) = \F_q^{\times}$ also (as $T_0$ is $\tau$-invariant), and so $T_0$ acts as the full group of scalars on each factor of a series $1=Q_0<Q_1<\cdots <Q_4=Q$ with $Q_i/Q_{i-1} \cong \F_q$ for all $i$; hence $d_{M_0}(Q) \leqs 4$ and it follows that $d(M) \leqs 4+d(T_0)+1 \leqs 7$, a contradiction.

Now assume $\pi_1(T_0) = A < \F_q^{\times}$. As $T_0$ is $\tau$-invariant, $\pi_2(T_0) = A$ as well, and so by maximality
$T_0 = \{h_\a(t)h_\b(u) \, : \, t,u\in A\}$. If $e$ is even (recall that $q = 2^e$) then (again by maximality) $|A|$ is divisible by $q^{1/2}-\e$ for some $\e = \pm 1$, and now the result follows as in the previous paragraph, using Lemma \ref{bbdfield}. On the other hand, if $e$ is odd, then the automorphism $t\mapsto t^2$ of $\F_q$ has odd order, so there is an automorphism $\phi$ of $\F_q$ such that $\phi^2(t) = t^2$ for all $t \in \F_q$. But then
\[
\langle \{h_\a(t)h_\b(u),\,h_\a(\phi(v))h_\b(v)\, : \,t,u\in A,\,v\in \F_q^{\times}\} \rangle
\]
is a proper $\tau$-invariant subgroup of $T$, contradicting the maximality of $T_0$.

Finally for this case ($G_0 = C_2(q)$), if  $Q \not \leqs M$ then $M_0 = (M\cap Q).T$ and $M\cap Q$ is a maximal $T$-invariant subgroup of $Q$; it follows that $d_{M_0}(M\cap Q) \leqs 3$ and so $d(M) \leqs 3+d(T)+1 \leqs 6$, a contradiction.

The case where $G_0 = G_2(3^e)$ and $G$ contains a graph or graph-field automorphism is handled in very similar fashion. The case $G_0 = {\rm L}_3(q)$ is also similar, but this time $\tau$ sends $h_\a(t) \mapsto h_\b(t^r)$,
$h_\b(u) \mapsto h_\a(u^{r})$ for all $t,u\in \F_q^{\times}$, and in the case of the above argument where $M_0 = QT_0$, we must have $\pi_i(T_0) = \F_q^{\times}$ for $i=1,2$, giving $d_{M_0}(Q) \leqs 3$.
\end{proof}

\begin{prop}\label{cla:main}
Theorem $\ref{max}$ holds in the case where $M<H<G$ with $G$ an almost simple group of Lie type and $H$ a maximal parabolic subgroup of $G$.
\end{prop}

\begin{proof}
Let $G_0$ denote the socle of $G$, which is a simple group of Lie type over $\F_q$, a field of characteristic $p$.

Let $M_0=M\cap G_0$, and write $H_0=H\cap G_0 = P = QR$, a parabolic subgroup with unipotent radical $Q$ and Levi subgroup $R$. We use the notation $P = P_{ij...}$ to mean a parabolic with excluded nodes $i,j,\ldots $ from the Dynkin diagram.

By Lemma \ref{borel}, we may assume that $H_0$ is not a Borel subgroup. In particular, $G_0$ is not of type $^2\!G_2$ or  $^2\!B_2$. We also exclude for now the cases where $(G_0,p)$ is {\it special} in the sense of \cite{ABS} -- that is to say, $p=2$ and $G_0$ is of type $C_n$, $F_4$, $^2\!F_4$, $G_2$,  or $p=3$ and $G_0$ is of type $G_2$.
We shall deal with these excluded cases at the end of the proof.

Suppose first that $G_0$ is untwisted and $H_0=P_i$ for some $i$. Then by \cite[Theorem 2(a)]{ABS}, $Q/Q'$ has the structure of an irreducible $\F_qR$-module, and $Q'\leqs \Phi(Q)$, so $Q'\leqs M_0$. It follows that either $M_0 = QK$ with $K$ a maximal $H/Q$-invariant subgroup of $R$, or $M_0 = Q'.R$.

Consider the case where $M_0=QK$. Now $R = R_0Z$, where $Z$ is a central torus of rank 1 inducing scalars on the module $Q/Q'$. Hence either $K = R_0Z_0$ with $Z_0<Z$, or $K = K_0Z$ with $K_0<R_0$. For $G_0$ classical, $R_0$ is of type ${\rm SL}_i(q)\times {\rm SL}_{n-i}(q)$ or ${\rm SL}_i(q) \times Cl_{n-2i}(q)$ and $Q/Q'$ is the corresponding tensor product space $U\otimes W$ with $\dim U = i$, $\dim W = n-i$ or $n-2i$ (here $ Cl_{n-2i}(q)$ denotes an appropriate classical group of dimension $n-2i$ over $\mathbb{F}_{q}$). Then using Lemma \ref{aff} we see that $Q/Q'$ is a cyclic $K$-module. Using Theorem \ref{t:bls} we deduce that $d(K) \leqs 6$.
Hence $d(M_0) \leqs 1+d(K) \leqs 7$. For $G_0$ of exceptional type, the irreducible module $Q/Q'$ has dimension at most 64 with equality for $(G_0,R_0) = (E_8(q),D_7(q))$, so we get $d(M_0) \leqs \dim (Q/Q')+d(K) \leqs 70$.

Now suppose $M_0 = Q'.R$. Here we bound $d(M_0)$ by $d_{M_0}(Q')+d(R)$. Now $R_0$ is a commuting product of at most 3 factors  which are either quasisimple or groups in $\{{\rm SL}_2(q),\O_3(q),\O_4^+(q)\, :\, q\leqs  3\}$; hence it is straightforward to check that $d(R) \leqs 4$. Also $d_{M_0}(Q')$ is at most the number of $R$-composition factors in $Q'$. By \cite[Theorem 2]{ABS}, this is 1 less than the $i$-th coefficient of the highest root in the root system of $G_0$, hence is at most 1 for $G_0$ classical, and at most 5 for $G_0$ exceptional. We conclude that $d(M_0) \leqs 9$ in this case.

Next assume that $G_0$ is twisted (and not special) -- hence of type $^2\!A_n$, $^2\!D_n$, $^2\!E_6$ or $^3\!D_4$.
In the first case consider the covering group $\hat G_0 = {\rm SU}_m(q)$ (where $m=n+1$), where $H_0 = P_i = QR$ with $R$ of type ${\rm SL}_i(q^2)\times {\rm SU}_{m-2i}(q)$. Here $Q/Q'$ has the structure of the $R$-module $V_1+V_2$ with $V_1 = U\otimes W$ and $V_2 = U^{(q)}\otimes W^*$, where $U,W$ are the natural modules for the factors of $R$. Hence as above, the possibilities for $M_0$ are
$QK$, $Q_1.R$ and $Q_2.R$, where $Q_i = Q'.V_i < Q$. We deal with the possibilities just as before. The $^2\!D_n$ or $^2\!E_6$ cases are very similar -- again, $Q/Q'$ is a sum of at most two irreducible $R$-submodules, leading to three possibilities for $M_0$ as above. Finally, if $G_0 = \,^3\!D_4(q)$ then $H_0=P_i$ with $i=1$ or 2. If $i=2$ then $R_0 = A_1(q^3)$ and $Q/Q'$ is the irreducible $\F_qR$-module $V_2\otimes V_2^{(q)}\otimes V_2^{(q^2)}$; and if $i=1$ then $R$ contains $A_1(q) \circ (q^3-1)$ and again $Q/Q'$ is an irreducible $\F_qR$-module (of dimension 6). In either case the result follows in the usual way.

The case where $G_0$ is of type $A_n$, $D_n$, $D_4$ or $E_6$ and $G$ contains a graph automorphism is very similar. In these cases, the maximal parabolics of $G$ for which $Q/Q'$ is a reducible $R$-module are $P_{i,n-i}$ (for $A_n$), $P_{n-1}$ (for $D_n$), $P_{134}$ (for $D_4$ when $G$ contains a triality automorphism) and $P_{16},P_{35}$ (for $E_6$). For these, \cite{ABS} shows that $Q/Q'$ is a sum of two irreducible $R$-modules (three for the $D_4$ case), and we argue as in the previous paragraph.

It remains to handle the cases where $G_0$ is special. These are dealt with by the same method as above. By the proof of \cite[Lemma 7.3]{BLS}, $Q/Q'$  has at most 4 $\F_qR$-module composition factors, so we can compute the possibilities for $M_0$ and bound $d(M_0)$ just as before.
\end{proof}

\vs

By combining this result with Propositions \ref{alt}, \ref{spor}, \ref{cla} and \ref{ex:main}, we conclude that the proof of Theorem \ref{max} is complete.

\begin{rmk} \label{tight}
{\rm The upper bound of 70 in part (ii) of Theorem \ref{max} is not sharp, and we make some remarks here about how one could go about improving it. As observed in the proof of Proposition \ref{cla:main}, we have this upper bound of 70 because of second maximal subgroups $M<QR = P_1$, a $D_7$-parabolic subgroup of $E_8(q)$, of the form $M = QK_0Z$ where $K_0$ is a maximal subgroup of $D_7(q)$. To improve the bound significantly, one would have to study the actions of such subgroups $K_0$ on
$Q/Q'$, which is a 64-dimensional spin module for $D_7(q)$.  Likewise, the $E_7$-parabolic $P_8$ of $E_8(q)$ has maximal subgroups $M = QK_0Z$ with $K_0$ a maximal subgroup in $E_7(q)$ (not all of which are known); consequently, in order to improve the obvious upper bound $d(M) \leqs \dim(Q/Q') + d(K_0) \leqs 60$ in this case, one would have to study the actions of such $K_0$ on the 56-dimensional $E_7(q)$-module $Q/Q'$.}
\end{rmk}

We are also in a position to give a proof of Theorem \ref{open}.

\vspace{2mm}

\noindent \textbf{Proof of Theorem \ref{open}.}

Clearly, part (i) of Theorem \ref{open} implies (ii), and (ii) implies (iii). For the next implication, note that the question \eqref{e:star} stated in Section \ref{s:intro} has a negative answer if and only if there exists a constant $c$ such that if $p$ is a prime and $(p^k-1)/(p^{\ell}-1)$ is prime for some natural numbers $k,\ell$, then $k\leqs c\ell$. Hence the fact that (iii) implies (iv) follows from Lemma \ref{arb}.

Finally, we show that (iv) implies (i). Assume (iv) holds, and let $G$ be an almost simple group with socle $G_0$. Let $M$ be second maximal in $G$. By Theorem \ref{max}, we have $d(M) \leqs 70$ except possibly if $G_0  = {\rm L}_2(q)$, $^2{}B_2(q)$ or $^2{}G_2(q)$, and $M$ is maximal in a Borel subgroup $B$ of $G$. In the latter cases, $B\cap G_0 = UT$ as in Lemma \ref{arb}. If $U\not \leqs M$ then $d(M) \leqs 10$ by Lemma \ref{lem1}; and if $U \leqs M$, then $d(M)$ is bounded by Lemma \ref{arb} together with the assumption (iv). Hence (iv) implies (i) and the proof of Theorem \ref{open} is complete.
 \hspace*{\fill}$\square$

\section{Random generation and third maximal subgroups}\label{s:final}

In this final section we prove Proposition \ref{unb} and Theorems \ref{nu} and \ref{poly}.

\vspace{2mm}

\noindent \textbf{Proof of Proposition \ref{unb}.}

Let $p\geqs 5$ be a prime such that $p\equiv \pm 3 \imod{8}$. The group ${\rm PGL}_2(p)$ has a maximal subgroup $S_4$ (cf. \cite{Dix}), and $S_{p+1}$ has a maximal subgroup ${\rm PGL}_2(p)$ (by \cite{LPS}). Moreover, for $n=2(p+1)$, the imprimitive subgroup $S_2 \wr S_{p+1}$ is maximal in $S_n$ (again by \cite{LPS}). Hence  we have the following chain of subgroups of $S_n$, each maximal in the previous one:
\[
S_n > S_2 \wr S_{p+1} > S_2 \wr {\rm PGL}_2(p) > (S_2)^{p+1}.S_4.
\]
Write $M = (S_2)^{p+1}.S_4$, and let $B$ be the base group $(S_2)^{p+1}$. By the Schreier index formula, $d(B)-1 \leqs |M:B|\,(d(M)-1)$, and hence
$$d(M) > \frac{d(B)-1}{24} = \frac{p}{24}.$$
Since $M$ is third maximal in $S_n$ and $p$ can be arbitrarily large, this completes the proof of the proposition.  \hspace*{\fill}$\square$

\vspace{4mm}

For the proof of Theorem \ref{nu}, we need the following result on chief factors of second maximal subgroups.

\begin{prop}\label{chief}
Let $M$ be a second maximal subgroup of an almost simple group. Then $\gamma(M) \leqs 5$, where $\gamma(M)$ is the number of non-abelian chief factors of $M$.
\end{prop}

\begin{proof}
Let $G$ be an almost simple group with socle $G_0$ and write $M<H<G$ with $M$ maximal in $H$ and $H$ maximal in $G$. Note that if $N$ is a normal subgroup of $M$, then $\gamma(M) \leqs \gamma(N)+\gamma(M/N)$. In particular, if $N$ is solvable then $\gamma(M) = \gamma(M/N)$. By \cite[Lemma 8.2]{BLS}, we have $\gamma(M)\leqs 3$ if $H$ is almost simple, so we may assume otherwise. More generally, if $H$ is of the form $H=N.A$, where $N$ is solvable and $A$ is almost simple, then either $M=(M \cap N).A$ and $\gamma(M) =1$, or $M=N.J$ and $J<A$ is maximal, so $\gamma(M) \leqs 3$. Similarly, if $H=N.(A \times B)$ with $N$ solvable and $A$ and $B$ almost simple, then either $\gamma(M)=2$ or $M=N.J$ with $J<A \times B$ maximal and it is easy to check that $\gamma(M)=\gamma(J) \leqs 4$.

If $G_0$ is sporadic then all the maximal subgroups of $G$ are known (apart from a handful of small almost simple candidates in the Monster) and it is straightforward to verify the bound $\gamma(M) \leqs 4$ by direct inspection.
Next suppose $G_0=A_n$ is an alternating group. As noted in the proof of  Proposition \ref{alt}, the possibilities for $H$ are determined by the O'Nan-Scott theorem and once again it is easy to check that $\gamma(M) \leqs 4$. This bound is sharp. For example, if $G=S_n$ and $H=S_k \wr S_t$, where $k \geqs 5$ and $t \geqs 11$, then $M=(S_k)^t.(S_5 \times S_{t-5})$ is a maximal subgroup of $H$ with $\gamma(M)=4$.

Next assume $G_0$ is a classical group. Here we use \cite{KL} to inspect the possibilities for $H$ (recall that we may assume $H$ is not almost simple) and one checks that $\gamma(M) \leqs 4$ if $H$ is non-parabolic. In fact, the same bound holds in all cases, with the possible exception of the case where $G_0 = {\rm L}_{n}(q)$ and $H$ is a parabolic subgroup of type $P_{m,n-m}$ as described in \cite[Proposition 4.1.22]{KL}. In the latter case, we could have $\gamma(M) = \gamma(J)+2$ where $J = K \cap {\rm L}_{a}(q)$ for some maximal subgroup $K$ of an almost simple group with socle ${\rm L}_{a}(q)$ (here $a=m$ or $n-m$). Therefore, $\gamma(M) \leqs 5$. Similar reasoning applies when $G_0$ is an exceptional group. A convenient description of the  maximal subgroups of $G$ is given in \cite[Theorem 8]{LS03} and it is straightforward to show that $\gamma(M) \leqs 5$.
\end{proof}

We now derive consequences concerning the invariant $\nu(M)$ defined in Section \ref{s:intro}. Our main tool is Theorem 1 of Jaikin-Zapirain and Pyber \cite{JP}.

\begin{cor}\label{all}
There exists an absolute constant $c$ such that if $M$ is a second maximal subgroup of an almost simple group, then $\nu(M) \leqs c\, d(M)$. Consequently $\nu(M)$ is bounded if and only if $d(M)$ is bounded.
\end{cor}

\begin{proof}  Let $\b$ be the constant in  \cite[Theorem 1]{JP}. By combining Proposition \ref{chief} with this theorem, we obtain
$$\nu(M) < \b d(M) + \frac{\log(\gamma(M))}{\log 5} \leqs \b d(M)+1.$$
The result follows.
\end{proof}

\vspace{2mm}

\noindent \textbf{Proof of Theorem \ref{nu}.}

Let $G$ be an almost simple group with socle $G_0$ and let $M$ be a second maximal subgroup of $G$ which is not as in part (iii) of Theorem \ref{max}. Then $d(M) \leqs 70$ by Theorem \ref{max}, and the result follows from Corollary \ref{all}. \hspace*{\fill}$\square$

\vspace{2mm}

For the proof of Theorem \ref{poly} we need the following result, which
may be of some independent interest.

\begin{lem}\label{submodules} Let $R$ be a finite-dimensional algebra over a finite field $\F$.
Let $M$ be an $R$-module of finite dimension over $\F$. Then $M$ has at most $|M/JM|-1$ maximal submodules,
where $J$ is the Jacobson radical of $R$. Moreover, this upper bound is best possible.
\end{lem}

\begin{proof}
It is well known that every maximal submodule of $M$ contains $JM$.
Therefore the number of maximal submodules of $M$ equals the number of
maximal submodules of $M/JM$ (as an $R/J$-module).
This enables us to reduce to the case where $J=0$, so that $R$
is a semisimple algebra and $M$ is a semisimple $R$-module.

Hence we may write
\[
M = \bigoplus_{i=1}^m n_iS_i,
\]
where the $S_i$ ($1 \leqs i \leqs m$) are pairwise non-isomorphic simple $R$-modules, and $n_i \geqs 1$ is the multiplicity of $S_i$.

Let $M_0 < M$ be a maximal submodule. Then $M/M_0 \cong S_i$ for some unique $i$ with $1 \leqs i \leqs m$. It follows that $M_0 \supseteq M_i$ where $M_i = \bigoplus_{j \ne i} n_jS_j$. Hence $M_0/M_i$ may be regarded as a maximal submodule of $n_iS_i$.

The number of such maximal submodules is less than
$|\Hom(n_iS_i,S_i)| = |\End(S_i)|^{n_i}$. Since $S_i$ (being simple) is a cyclic
module we have $|\End(S_i)| \leqs |S_i|$. It follows that $M$ has less than $|S_i|^{n_i}$ maximal submodules $M_0$ satisfying $M/M_0 \cong S_i$. Summing over $i$ we see that the number of maximal submodules of $M$ is less than
\[
\sum_{i = 1}^m |S_i|^{n_i} \leqs \prod_{i=1}^m |S_i|^{n_i} = |M|.
\]
This completes the proof of the upper bound.

To show that this upper bound is best possible, let $R = \F = \F_2$ and let $M$ be
a $d$-dimensional vector space over $\F$. Then $|M/JM| = 2^d$ and $M$ has $|M/JM|-1$ maximal submodules.
\end{proof}

\vspace{2mm}

\noindent \textbf{Proof of Theorem \ref{poly}.}

Let $G$ be an almost simple group with socle $G_0$.
By \cite[Corollary 6]{BLS}, $G$ has at most $n^a$ second maximal subgroups of index $n$ for some absolute constant $a$ and for all $n \geqs 1$. It therefore suffices to show the following.

\vspace{2mm}

\noindent
\bf{Claim.} {\it There is an absolute constant $b$ such that, for every $n \geqs 1$,
every second maximal subgroup $M$ of $G$ has at most $n^b$ maximal subgroups of index $n$ in $G$.}

\vspace{2mm}

Indeed, assuming the claim, a third maximal subgroup $N$ of index $n$ in $G$ is contained in some
second maximal subgroup $M$ of $G$, which -- being of index at most $n$ -- can be chosen
in at most $n^{a+1}$ ways. Given $M$, the third maximal subgroup $N$ can be chosen in at most
$n^b$ ways. Thus $G$ has at most $n^{a+b+1}$ third maximal subgroups of index $n$.

To prove the claim, let $M$ be a second maximal subgroup of $G$.
Recall that $m_n(M)$ denotes the number of maximal subgroups of $M$ of index $n$ in $M$.
If $m_n(M) \leqs n^b$ for an absolute constant $b$ and for all $n$ then the claim follows
immediately.

We show that this is the case assuming $G_0$ is not ${\rm L}_2(q)$, ${}^2B_2(q)$ or ${}^2G_2(q)$.
Indeed, in this case we have $\nu(M) \leqs c$ by Theorem \ref{nu}, so by \cite[Proposition 1.2]{lub}
we have $m_n(M) \leqs n^b$ where $b = c + 3.5$.

Now assume that $G_0 = {\rm L}_2(q)$, ${}^2B_2(q)$ or ${}^2G_2(q)$. We apply Lemma \ref{borel}
which describes the second maximal subgroups $M$ of $G$ for which $d(M)$ is possibly unbounded. By
Corollary \ref{all} these are the ones for which $\nu(M)$ is possibly unbounded.

Suppose $G_0 = {\rm L}_2(q)$ with $q=p^k$, and let $G = G_1.A$ where $G_1 = G\cap {\rm PGL}_2(q)$ and $A$ is a group of field automorphisms of order dividing $k$. Set $f = |A|$.
The relevant second maximal subgroups $M$ are of the form $U.T_1.A$, where $U \cong \F_q$ and $T_1 \leqs \F_q^{\times}$
has order $e$.
Let $\ell$ be the multiplicative order of $p$ modulo $e$ as in Lemma \ref{arb}. Note that $|G:M| > q$.
We shall show that such subgroups $M$ have less than $n^4$ maximal subgroups of index $n$ in $G$.

The maximal subgroups of such a group $M$ split naturally into two types.
The first type is $U.X$ where $X$ is maximal in $T_1.A$. Now, $T_1.A$ is metacyclic, and so are its
subgroups. Since all subgroups of $T_1.A$ are $2$-generated, there are at most $|T_1.A|^2 = e^2 f^2 < q^2 k^2 < q^4$
such subgroups (including non-maximal ones). This proves the claim with $b=4$ for subgroups of $M$ of the first type.

The second type of subgroups of $M$ is $U_0.T_1.A$, where $U_0$ is a proper $\F_{p^\ell}$-subspace of $U \cong \F_q$ that is maximal $A$-invariant. Let $q_0 = p^{\ell}$ and consider the group algebra $R = \F_{q_0}[A]$. Then $U, U_0$ are $R$-modules and $U_0$ is a maximal submodule of $U$.

Applying Lemma \ref{submodules}, there are fewer than $|U| = q$ possibilities for $U_0$.
We now claim that, given $U_0$, there are less than $q^3$ subgroups of $M$ of type $U_0.T_1.A$.
Indeed, $T_1$ is split in $U_0.T_1$, so there are less than $q$ possibilities for $U_0.T_1$;
and the cyclic group $A$ is generated by an element of the form $u \phi$ where $u \in U_0.T_1$ and
$\phi$ is a fixed field automorphism, so there are less than $q^2$ possibilities for such a generator,
hence less than $q^3$ possibilities in all for $U_0.T_1.A$.

We conclude that the number of maximal subgroups of $M$ of the second type is also less than $q^4$.
Since $|G:M| > q$ this completes the proof of the claim for $G_0 = {\rm L}_2(q)$, with $b=4$.

The proofs for Suzuki and Ree groups are similar, and this completes the proof of the claim, and hence of the theorem. \hspace*{\fill}$\square$


\newpage

\begin{thebibliography}{99}

\bibitem{asch2}
M. Aschbacher, \emph{On intervals in subgroup lattices of finite groups},
J. Amer. Math. Soc. \textbf{21} (2008), 809--830.

\bibitem{asch} M. Aschbacher, \emph{On the maximal subgroups of the finite classical groups},
  Invent. Math. \textbf{76} (1984), 469--514.

\bibitem{AG} M. Aschbacher and R. Guralnick,
\emph{Some applications of the first cohomology group},
J. Algebra \textbf{90} (1984), 446--460.


\bibitem{ABS}
H. Azad, M. Barry and G.M. Seitz, \emph{On the structure of parabolic
  subgroups}, Comm. Algebra \textbf{18} (1990), 551--562.

\bibitem{Basile} A. Basile, \emph{Second maximal subgroups of the finite alternating and symmetric groups}, PhD thesis (Australian National University, 2001), arxiv:0810.3721.

\bibitem{Magma} W. Bosma, J. Cannon and C. Playoust,
\emph{The {\sc Magma} algebra system I: The user language},
J. Symbolic Comput. \textbf{24} (1997), 235--265.

 \bibitem{BHR} J.N. Bray, D.F. Holt and C.M. Roney-Dougal, \emph{The {M}aximal {S}ubgroups of the {L}ow-dimensional {F}inite {C}lassical {G}roups}, London Math. Soc. Lecture Note Series, vol. 407, Cambridge University Press, 2013.

\bibitem{BLS} T.C. Burness, M.W. Liebeck and A. Shalev, \emph{Generation and random generation: From simple groups to maximal subgroups}, Advances in Math. \textbf{248} (2013), 59--95.

\bibitem{BOW}
T.C. Burness, E.A. O'Brien, and R.A. Wilson, \emph{Base sizes for sporadic
  simple groups}, Israel J. Math. \textbf{177} (2010), 307--334.

\bibitem{car} R.W. Carter, \emph{Simple groups of Lie type},
John Wiley and Sons, 1972.

\bibitem{CLSS}
A.M. Cohen, M.W. Liebeck, J.~Saxl, and G.M. Seitz, \emph{The local maximal
  subgroups of exceptional groups of {L}ie type}, Proc. London Math. Soc.
  \textbf{64} (1992), 21--48.

\bibitem{Atlas}
J.H. Conway, R.T. Curtis, S.P. Norton, R.A. Parker, and R.A. Wilson, \emph{Atlas of
  {F}inite {G}roups}, Oxford University Press, 1985.

\bibitem{DL} F. Dalla Volta and A. Lucchini, \emph{Generation of almost simple groups}, J. Algebra \textbf{178} (1995), 194--223.

\bibitem{Dix} L.E. Dickson, \emph{Linear groups with an exposition of the Galois field theory}, Teubner, Leipzig 1901 (Dover reprint 1958).

\bibitem{feit83}
W. Feit, \emph{An interval in the subgroup lattice of a finite group which is isomorphic to $M_7$}, Algebra Universalis \textbf{17} (1983), 220--221.

\bibitem{JP}
A. Jaikin-Zapirain and L. Pyber, \emph{Random generation of finite and profinite groups and group enumeration}, Annals of Math. \textbf{173} (2011), 769--814.

\bibitem{KanL}
W.M. Kantor and A. Lubotzky, \emph{The probability of generating a finite classical group}, Geom. Dedicata \textbf{36} (1990), 67--87.


\bibitem{K}
P.B. Kleidman, \emph{The maximal subgroups of the finite 8-dimensional
  orthogonal groups {${\rm P\O}_{8}^{+}(q)$} and of their automorphism groups},
  J. Algebra \textbf{110} (1987), 173--242.

\bibitem{KL}
P.B. Kleidman and M.W. Liebeck, \emph{The {S}ubgroup {S}tructure of the
  {F}inite {C}lassical {G}roups}, London Math. Soc. Lecture Note Series, vol.
  129, Cambridge University Press, 1990.

\bibitem{LMSh} M.W. Liebeck, B.M.S. Martin and A. Shalev,
\emph{On conjugacy classes of maximal subgroups of finite simple groups,
and a related zeta function}, Duke Math. Journal \textbf{128} (2005), 541--557.

\bibitem{LPS} M.W. Liebeck, C.E. Praeger and J. Saxl, \emph{A classification of the maximal subgroups of the finite alternating and symmetric groups} J. Algebra \textbf{111} (1987),  365--383.

\bibitem{LSS}
M.W. Liebeck, J.~Saxl, and G.M. Seitz, \emph{Subgroups of maximal rank in
  finite exceptional groups of {L}ie type}, Proc. London Math. Soc. \textbf{65}
  (1992), 297--325.

\bibitem{LS10}
M.W. Liebeck and G.M. Seitz, \emph{Maximal subgroups of exceptional groups of
  {L}ie type, finite and algebraic}, Geom. Dedicata \textbf{36} (1990),
  353--387.

\bibitem{LS03}
M.W. Liebeck and G.M. Seitz, \emph{A survey of of maximal subgroups of exceptional groups of Lie type}, in Groups, combinatorics \& geometry (Durham, 2001), 139--146, World Sci. Publ., 2003.

 \bibitem{LSh95}
M.W. Liebeck and A. Shalev, \emph{The probability of generating a finite simple group}, Geom. Dedicata \textbf{56} (1995),
  103--113.

\bibitem{lub}
A. Lubotzky, \emph{The expected number of random elements to generate a finite group}, J. Algebra \textbf{257} (2002), 452--459.

\bibitem{Ma} A. Mann, \emph{Positively finitely generated groups}, Forum Math. \textbf{8} (1996), 429--459.

  \bibitem{MS}
  A. Mann and A. Shalev, \emph{Simple groups, maximal subgroups, and probabilistic aspects of profinite groups}, Israel J. Math. \textbf{96} (1996), 449--468.

\bibitem{Pak}
I. Pak, \emph{On probability of generating a finite group}, preprint (1999).

\bibitem{Pal}
P.P. P\'{a}lfy, \emph{On Feit's examples of intervals in subgroup lattices},
J. Algebra \textbf{116} (1988), 471--479.

\bibitem{St}
R. Steinberg, \emph{Generators for simple groups}, Canad. J. Math.
  \textbf{14} (1962), 277--283.

\bibitem{Suz}
M. Suzuki, \emph{On a class of doubly transitive groups}, Annals of Math.
  \textbf{75} (1962), 105--145.

\bibitem{Thev} J. Th\'{e}venaz, \emph{Maximal subgroups of direct products}, J. Algebra \textbf{198} (1997), 352--361.

\bibitem{Ward}
H.N. Ward, \emph{On {R}ee's series of simple groups}, Trans. Amer. Math. Soc.
  \textbf{121} (1966), 62--89.

\bibitem{WebAt}
R.A. Wilson et~al., \emph{A {W}orld-{W}ide-{W}eb {A}tlas of finite group
  representations}, {\texttt{http://brauer.maths.qmul.ac.uk/Atlas/v3/}}.

\end{thebibliography}
\end{document}